\documentclass[11pt]{amsart}

\usepackage{amsmath}
\usepackage{amstext}
\usepackage{amssymb}
\usepackage{amsthm}
\usepackage{enumerate}

\makeatletter
\def\imod#1{\allowbreak\mkern10mu({\operator@font mod}\,\,#1)}
\makeatother

\newtheorem{theorem}{Theorem}[section]

\newtheorem{prop}[theorem]{Proposition}
\newtheorem{lemma}[theorem]{Lemma}
\newtheorem{claim}[theorem]{Claim}

\newtheorem*{clm}{Claim}

\theoremstyle{definition}
\newtheorem{definition}[theorem]{Definition}

\newtheorem{corollary}[theorem]{Corollary}
\theoremstyle{remark}
\newtheorem{remark}[theorem]{Remark}

\theoremstyle{remark}

\numberwithin{equation}{section}

    \DeclareMathOperator{\dom}{dom}
    
    \DeclareMathOperator{\ran}{ran}

    \DeclareMathOperator{\im}{im}

    \DeclareMathOperator{\fix}{fix}
    \DeclareMathOperator{\cofin}{cofin}
    
    \DeclareMathOperator{\card}{card}
    \DeclareMathOperator{\sign}{sign}
    \DeclareMathOperator{\oc}{oc}
    \DeclareMathOperator{\cl}{cl}
    \DeclareMathOperator{\Dp}{Dp}
    \DeclareMathOperator{\Rk}{Rk}
    \DeclareMathOperator{\non}{non}
    \DeclareMathOperator{\add}{add}
    \DeclareMathOperator{\cof}{cof}

\newcommand{\srestrict}{\, {|\:\!\!\!\!\upharpoonright}\,}

\newcommand{\restrict}{{\upharpoonright}}
\newcommand{\rest}{{\upharpoonright}}

    \newcommand{\forces}{\Vdash}
    \newcommand{\What}[1]{\widehat W_{\!\! #1}}
    \newcommand{\halts}{\!\downarrow}
    \newcommand{\undef}{\!\uparrow}
    \newcommand{\on}{\mathbb{ON}}
    \newcommand{\la}{\langle}
    \newcommand{\ra}{\rangle}
    \newcommand{\PO}{\mathcal{P}}

\def\Z{{\mathbb Z}}

\def\F{{\mathbb F}}
\def\P{{\mathbb P}}

\def\Q{{\mathbb Q}}
\def\BbS{{\mathbb S}}
\def\BbL{{\mathbb L}}
\def\PP{{\mathbb P}}
\def\QQ{{\mathbb Q}}
\def\D{{\mathbb D}}

\date{\today}
\begin{document}

\title[Template iterations and maximal cofinitary groups]{Template iterations and maximal cofinitary groups}

\author{Vera Fischer}

\address{Kurt G\"odel Research Center, University of Vienna, W\"ahringer Strasse 25, 1090 Vienna, Austria}
\email{vera.fischer@univie.ac.at}

\author{Asger T\"ornquist}

\address{Department of Mathematical Sciences, University of Copenhagen, Universitetspark 5, 2100 Copenhagen, Denmark}
\email{asgert@math.ku.dk}

\thanks{The first author would like to thank the Austrian Science Fund FWF for the generous support
through grant no. M1365-N13. The second author would like to thank Denmark's Council for Independent Research for generous support through grant. no. 10-082689/FNU}



\maketitle

\begin{abstract} In \cite{brendle03}, J\"org Brendle used Hechler's forcing notion for adding a maximal almost family along an appropriate
template forcing construction to show that  $\mathfrak a$ (the minimal size of a maximal almost disjoint family) can be of countable cofinality. The main result of the present paper is that  $\mathfrak a_g$, the minimal size of maximal cofinitary group, can be of countable cofinality. To prove this we define a natural poset for adding a maximal cofinitary group of a given cardinality, which enjoys certain combinatorial properties allowing it to be used within a similar template forcing construction. Additionally we obtain that $\mathfrak a_p$, the minimal size of a maximal family of almost disjoint permutations, and $\mathfrak a_e$, the minimal size of a maximal eventually different family, can be of countable cofinality.
\end{abstract}

\section{Introduction}

In~\cite{shelah}, Shelah introduced a template iteration forcing technique, which provided the consistency of  $\aleph_2\leq\mathfrak{d}<\mathfrak{a}$
(without the assumption of a measurable). The technique was further developed by Brendle, who established that it is consistent that the almost disjointness number $\mathfrak a$ is of countable cofinality (see~\cite{brendle03}). Broadly speaking, the template iteration of~\cite{brendle03} can be thought of as a forcing construction, which on one side has characteristics of a ``product-like'' forcing, and on the other hand has characteristics of finite support iteration.  In \cite{brendle03}, the ``product-like'' side of the construction was used to force a maximal almost disjoint family of some arbitrary uncountable cardinality $\lambda$, which in particular can be of countable cofinality, while the ``finite support'' side of the construction was used to add a cofinal family of dominating reals of a prescribed size $\lambda_0$. This cofinal family gives a prescribed size of the bounding number in the generic extension and so gives a prescribed lower bound of $\mathfrak{a}$. An isomorphism of names argument (which assumes CH holds in the ground model) provides that there are no mad families of intermediate cardinalities $\mu$, i.e. cardinality $\mu$ such that $\aleph_2\leq\lambda_0\leq\mu<\lambda$.

A \emph{cofinitary group} is a subgroup $G$ of the group $S_\infty$ of all permutations of $\omega$, which has the property that each of its non-identity elements has only finitely many fixed point. Such a group is called \emph{maximal} if it is not contained in a strictly larger cofinitary group. The minimal size of a maximal cofinitary groups is denoted $\mathfrak{a}_g$. Following an approach, similar to the one of~\cite{brendle03}, we prove:

\begin{theorem}\label{mainthm1} Assume CH. Let $\lambda$ be a singular cardinal of countable cofinality. Then there is a ccc generic extension in
which  $\mathfrak{a}_g=\lambda$.
\end{theorem}

To prove the above theorem, we introduce a forcing notion which adds a maximal cofinitary group of prescribed size, and which enjoys certain combinatorial properties, allowing for the poset to be iterated along a template (see Definition~\ref{d.adding_mcg}). We use this poset along the ``product-like'' side of an appropriate template iteration, in order to add a maximal cofinitary group of desired cardinality, say $\lambda$. The ``finite support" side of the construction is used to add a cofinal family $\Phi$ of slaloms, each of which localizes the ground model reals. Using a combinatorial characterization of $\add(\mathcal{N})$ and $\cof(\mathcal{N})$ (the additivity and cofinality of the null ideal) due to Bartoszy\'nski, we obtain that in the final generic extension both of those cardinal invariant have the size of the family $\Phi$. By a result of Brendle, Spinas and Zhang (see~\cite{brspzh00}), the uniformity of the meagre ideal $\non(\mathcal{M})$ is less than or equal to the $\mathfrak a_g$, and so we obtain that in the final generic extension $|\Phi|$ is a lower bound for $\mathfrak{a}_g$. Finally, an isomorphism of names argument, which is almost identical to the maximal almost disjoint families case, provides that in the final generic extension there are no maximal cofinitary groups of intermediate cardinalities, i.e. cardinalities $\mu$ such that $|\Phi|\leq\mu<\lambda$.
Again for this isomorphism of names argument to work we have to assume that CH, as well as $\aleph_2\leq|\Phi|$.

Though proving Theorem \ref{mainthm1} is our main goal, we take an axiomatic approach which allows us to obtain slightly more. We define two classes of forcing notions which in a natural capture the key properties of our poset for adding a maximal cofinitary group and Hechler's forcing notion for adding a dominating real, respectively. We refer to these posets as \emph{finite function posets with the strong embedding property} (see~Definitions~\ref{finite_function_poset} and~\ref{d.strong_reduction})) and \emph{good $\sigma$-Suslin forcing notions} (see Definitions~\ref{d.sigma_suslin} and~\ref{d.good_suslin}) respectively. We generalize the template iteration techniques of~\cite{brendle03}, so that arbitrary representatives of the above two classes can be iterated along a template (see~Definition~\ref{template_poset} and Lemma~\ref{l.mainlemma}) and establish some basic combinatorial properties of this generalized iteration.
Whenever $\mathcal{T}$ is a template, $\Q$ is a finite function poset with the strong embedding property, and $\BbS$ is a good  $\sigma$-Suslin forcing notion, we denote by $\P(\mathcal{T},\Q,\BbS)$ the iteration of $\Q$ and $\BbS$ along $\mathcal{T}$ (see Definition~\ref{template_poset}). For example we show that whenever $\Q$ is Knaster, then the entire iteration $\P(\mathcal{T},\Q,\BbS)$ is Knaster (see Lemma~\ref{l.template_knaster}).

Following standard notation, let $\mathfrak{a}_p$ and $\mathfrak{a}_e$ denote the minimal size of a maximal family of almost disjoint permutations on $\omega$ and the minimal size of a maximal almost disjoint family of functions from $\omega$ to $\omega$, respectively.  Let $\mathcal{T}_0$ be the template used by Brendle in~\cite{brendle03}. Then our results can be summarized as follows:

\begin{theorem}\label{mainthm2} Assume CH. Let $\lambda$ be a singular cardinal of countable cofinality and let  $\bar{\mathfrak{a}}\in\{\mathfrak{a},\mathfrak{a}_p,\mathfrak{a}_g,\mathfrak{a}_e\}$. Then there are a good $\sigma$-Suslin poset $\BbS_{\bar{\mathfrak{a}}}$
and a finite function poset with the strong embedding property $\Q_{\bar{\mathfrak{a}}}$, which is Knaster (and so by Lemma~\ref{l.template_knaster}  $\P(\mathcal{T}_0,\Q_{\bar{\mathfrak{a}}},\BbS_{\bar{\mathfrak{a}}})$ is Knaster)
such that $V^{\P(\mathcal{T}_0,\Q_{\bar{\mathfrak{a}}},\BbS_{\bar{\mathfrak{a}}})}\vDash {\bar{\mathfrak{a}}}=\lambda$. Then in particular $V^{\P(\mathcal{T}_0,\Q_{\bar{\mathfrak{a}}},\BbS_{\bar{\mathfrak{a}}})}\vDash\hbox{cof}(\bar{\mathfrak{a}})=\omega$.
\end{theorem}

The most interesting case is the maximal cofinitary groups case. In fact, for each $\bar{\mathfrak a}\in\{\mathfrak a_p,\mathfrak a_e\}$, the forcing notion $\Q_{\bar{\mathfrak{a}}}$ is closely related to the forcing notion for adding a maximal cofinitary group of arbitrary cardinality, presented in section \S 2.

\medskip

{\it Organization of the paper.} In \S 2, we introduce and study a forcing notion $\Q_{A,\rho}$ for adding a maximal cofinitary group with a generating set indexed by some given uncountable set $A$. In \S 3, we introduce the classes of good $\sigma$-Suslin forcing notions and finite function posets with the strong embedding properties. We define the template iteration $\P(\mathcal{T},\Q,\BbS)$ of arbitrary representatives $\BbS$ and $\Q$ of the above two classes respectively, along a given template $\mathcal{T}$ and show that $\P(\mathcal{T},\Q,\BbS)$ is a forcing notion. In \S 4, we establish some basic combinatorial properties of this generalized iteration. Theorem \ref{mainthm1} is proved in \S 5, and Theorem \ref{mainthm2} is proved in \S 6.

\section{A generalization of Zhang's forcing}

In~\cite{zhang1}, Zhang introduced a ccc forcing $\mathbb G_{H}$, where $H$ is a given cofinitary group in the ground model, such that forcing with $\mathbb G_H$ adds a permutation $f\in S_\infty$ such that the group $\langle H,f\rangle$ generated by $H$ and $f$ is cofinitary.

In this section we introduce a generalization of Zhang's forcing which adds, in one step, a cofinitary group of size $\kappa$ to the generic extension. While the results immediately obtained by doing this also could be achieved by an iteration of Zhang's forcing (see e.g. \cite{zhang1}), the template forcing we develop in the next section relies crucially on the forcing notion we define here.

We begin by giving several basic definitions and fixing notation.

\begin{definition}
1.  Let $A$ be a set. We denote by $W_A$ the set of reduced words in the alphabet $\langle a^{i}:a\in A, i\in\{-1,1\}\rangle$. The {\it free group} on generator set $A$ is the group $\F_A$ we obtain by giving $W_A$ the obvious concatenate-and-reduce operation. When $A=\emptyset$ then $\F_A$ is by definition the trivial group. Note that $A$ can be naturally identified with a subset of $\F_A$ which generates $\F_A$, and every function $\rho:B\to G$, where $G$ is any group, extends to a group homomorphism $\hat\rho: \F_B\to G$.

2. We denote by $\What{A}$ the set of all $w\in W_A$ such that either $w=a^n$ for some $a\in A$ and $n\in\Z\setminus\{0\}$, or $w$ starts and ends with a different letter. In the latter case, this means that there is $u\in W_A$, $a,b\in A$, $a\neq b$, and $i,j\in\{-1,1\}$ such that $w=a^{i}ub^{j}$ {\it without cancelation}. Note that any word $w\in W_A$ can be written as $w=u^{-1}w'u$ for some $w'\in \What{A}$ and $u\in W_A$.

3. For a (partial) function $f:\omega\to\omega$, let
$$
\fix(f)=\{n\in\omega:f(n)=n\}.
$$
We denote by $\cofin(S_\infty)$ set of cofinitary permutations in $S_\infty$, i.e. permutations $\sigma\in S_\infty$ such that $\fix(\sigma)$ is finite.

4. For a group $G$, a {\it cofinitary representation} of $G$ is a homomorphism $\varphi:G\to S_\infty$ such that $\im(\varphi)\subseteq \{I\}\cup\cofin(S_\infty)$. If $B$ is a set and $\rho:B\to S_\infty$ we say that $\rho$ {\it induces} a cofinitary representation of $\F_B$ if the canonical extension of $\rho$ to a homomorphism $\hat\rho:\F_B\to S_\infty$ is a cofinitary representation of $\F_B$.
\end{definition}

5. Let $A$ be a set and let $s\subseteq A\times\omega\times\omega$. For $a\in A$, let
$$
s_a=\{(n,m)\in\omega\times\omega: (a,n,m)\in s\}.
$$
For a word $w\in W_A$, define the relation $e_w[s]\subseteq \omega\times\omega$ recursively by stipulating that for $a\in A$, if $w=a$ then $(n,m)\in e_w[s]$ iff $(n,m)\in s_a$, if $w=a^{-1}$ then $(n,m)\in e_w[s]$ iff $(m,n)\in s_a$, and if $w=a^i u$ for some word $u\in W_A$ and $i\in\{1,-1\}$ without cancellation then
$$
(n,m)\in e_w[s]\iff (\exists k) e_{a^i}[s](k,m)\wedge e_u[s](n,k).
$$
If $s_a$ is a partial injection defined on a subset of $\omega$ for all $a\in A$, then $e_w[s]$ is always a partial injection defined on some subset of $\omega$, and we call $e_w[s]$ the {\it evaluation} of $w$ given $s$. By definition, let $e_{\emptyset}[s,\rho]$ be the identity in $S_\infty$.

6. If $s\subseteq A\times\omega\times\omega$ is such that $s_a$ is always a partial injection, and $w\in\ W_A$, then we will write $e_w[s](n)\halts$ when $n\in\dom(e_w[s])$, and $e_w[s](n)\undef$ when $n\notin \dom(e_w[s])$.

7. Finally, let $A$ and $B$ be disjoint sets and let $\rho:B\to S_\infty$ be a function. For a word $w\in W_{A\cup B}$ and $s\subseteq A\times\omega\times\omega$, we define
$$
(n,m)\in e_w[s,\rho]\iff (n,m)\in e_w[s\cup \{(b,k,l):\rho(b)(k)=l\}].
$$
If $s_a$ always is a partial injection for $a\in A$, then $e_w[s,\rho]$ is also a partial injection, and we call it the {\it evaluation} of $w$ given $s$ and $\rho$. The notations $e_w[s,\rho]\halts$ and $e_w[s,\rho]\undef$ are defined as before.

\medskip

The following lemma is obvious from the definitions. It will be used again and again, often without explicit mention.

\begin{lemma}\label{nocancel}
Fix sets $A$ and $B$ such that $A\cap B=\emptyset$, and a function $\rho:B\to S_\infty$. Let $w\in W_{A\cup B}$ and $s\subseteq A\times\omega\times\omega$ such that $s_a$ is a partial injection for all $a\in A$. Suppose $w=uv$ without cancellation for some $u,v\in W_{A\cup B}$. Then $n\in\dom(e_w[s,\rho])$ if and only if $n\in\dom(e_v[s,\rho])$ and $e_v[s,\rho](n)\in \dom(e_u[s,\rho])$. If moreover $w\in\What{A\cup B}$ then $n\in\fix(e_{w}[s,\rho])$ if and only $e_{v}[s,\rho](n)\in\fix(e_{vu}[s,\rho])$. In particular, $\fix(e_w[s,\rho])$ and $\fix(e_{vu}[s,\rho])$ have the same cardinality.
\end{lemma}

\begin{remark}
Note that if $w=uv$ {\it with} cancelation, or $w\notin \What{A\cup B}$, the above lemma may fail.
\end{remark}

\begin{definition}\label{d.adding_mcg}
Fix sets $A$ and $B$ such that $A\cap B=\emptyset$ and a function $\rho:B\to S_\infty$ such that $\rho$ induces a cofinitary representation $\hat\rho:\F_B\to S_\infty$. We define the forcing notion $\Q_{A,\rho}$ as follows:
\begin{enumerate}[(1)]
\item Conditions of $\Q_{A,\rho}$ are pairs $(s,F)$ where $s\subseteq A\times\omega\times\omega$ is finite and $s_a$ is a finite injection for every $a\in A$, and $F\subseteq \What{A\cup B}$ is finite.

\item $(s,F)\leq_{\Q_{A,\rho}} (t,E)$ if and only if $s\supseteq t$, $F\supseteq E$ and for all $n\in\omega$ and $w\in E$, if $e_w[s,\rho](n)=n$ then already $e_w[t,\rho](n)\halts$ and $e_w[t,\rho](n)=n$.
\end{enumerate}
If $B=\emptyset$ then we write $\Q_A$ for $\Q_{A,\rho}$.
\end{definition}

\begin{remark}

When $A$, $B$ and $\rho:B\to S_\infty$ are clear from the context, we may write $\leq$ instead $\leq_{\Q_{A,\rho}}$.

\end{remark}

Unless otherwise stated, we now always assume that $A$ and $B$ are disjoint sets, $A\neq\emptyset$ and $\rho:B\to S_\infty$ induces a cofinitary representation of $\F_B$.

\begin{lemma}\label{ccc}
The poset $\Q_{A,\rho}$ has the Knaster property.
\end{lemma}
\begin{proof}
For $w\in W_{A\cup B}$, write $\oc(w)$ for the (finite) set of letters occurring in $w$, and for $F\subseteq W_{A\cup B}$ let $\oc(F)=\bigcup_{w\in F}\oc(w)$. For $C\subseteq A\cup B$ and $w$ and $F$ as before, let $\oc_C(w)=\oc(w)\cap C$ and $\oc_C(F)=\oc(F)\cap C$. For $s\subseteq A\times\omega\times\omega$
let $\dom(s_\alpha)=\{a: \exists n,m\in\omega (a,n,m)\in s\}$.

Suppose that $\langle (s^\alpha, F^\alpha)\in \Q_{A,\rho}: \alpha<\omega_1\rangle$ is a sequence of conditions. By applying the $\Delta$-system Lemma \cite[Theorem 1.5]{kunen80} repeatedly we may assume that there are $A_0,A_1\subseteq A$ finite and $t\subseteq A\times\omega\times\omega$ finite such that for all $\alpha\neq\beta$, $s^{\alpha}\cap s^{\beta}=t$, $\oc_A(F^\alpha)\cap\oc_A(F^\beta)=A_0$ and
$$
(\dom s^{\alpha}\cup\oc_A(F^\alpha))\cap(\dom s^{\beta}\cup\oc_A(F^\beta))=A_1.
$$
Note that $\dom(t)$ and $A_0$ are subsets of $A_1$. Further, we may assume that $s^{\alpha}\cap A_1\times\omega\times\omega=t$, since this must be true for uncountably many $\alpha$ as $A_1$ is finite. Note then that $(s^\alpha\cup s^\beta,F^\alpha\cup F^\beta)\in \Q_{A,\rho}$ and that if $\alpha\neq\beta$ then
\begin{equation}\label{ccceq}
s^\alpha\cap \oc(F^\beta)\times\omega\times\omega\subseteq t.
\end{equation}
We claim that $(s^\alpha\cup s^\beta,F^\alpha\cup F^\beta)\leq_{\Q_{A,\rho}} (s^\beta,F^\beta)$. For this, suppose that $w\in F^\beta$ and that $e_w[s^\alpha\cup s^\beta,\rho](n)=n$. Then by \ref{ccceq} we have $e_w[t\cup s^\beta,\rho](n)=n$ and so $e_w[s^\beta,\rho](n)=n$. The proof that $(s^\alpha\cup s^\beta,F^\alpha\cup F^\beta)\leq_{\Q_{A,\rho}} (s^\alpha,F^\alpha)$ is similar.
\end{proof}

Let $G$ be $\Q_{A,\rho}$ generic (over $V$, say.) We define $
\rho_G:A\cup B\to S_\infty$ by
\begin{equation}\label{rho_G}
\rho_G(x)=\left\{\begin{array}{ll} \rho(x) & \text{ if } x\in B\\
\bigcup\{s_x:(\exists F\in \What{A\cup B})\ (s,F)\in G\} & \text{ if } x\in A.
\end{array}\right.
\end{equation}
We will see that $\rho_G$ induces a cofinitary representation of $A\cup B$ which extends $\rho$.
Of course, we first need to check that when $G$ is generic then
$$
\bigcup\{s_x:(\exists F\in \What{A\cup B})\ (s,F)\in G\}
$$
is a permutation. This is the content of the next Lemma, which is parallel to \cite[Lemma 2.2]{zhang1}.

\begin{lemma}\label{extension}
Let $A$ and $B$ be disjoint sets and $\rho:B\to S_\infty$ a function inducing a cofinitary representation of $\F_B$. Then

1. (``Domain extension'') For any $(s,F)\in \Q_{A,\rho}$, $a\in A$ and $n\in\omega$ such that $n\notin\dom(s_a)$ there are cofinitely many $m\in\omega$ s.t. $(s\cup\{(a,n,m)\},F)\leq (s,F)$.

2. (``Range extension'') For any $(s,F)\in \Q_{A,\rho}$, $a\in A$ and $m\in\omega$ such that $m\notin\ran(s_a)$ there are cofinitely many $n\in\omega$ s.t. $(s\cup\{(a,n,m)\},F)\leq (s,F)$.
\end{lemma}

We will first prove a slightly stronger version of this, but at first only for certain special ``good'' words.

\begin{definition}
Let $a\in A$ and $j\geq 1$. A word $w\in W_{A\cup B}$ is called \emph{$a$-good} of \emph{rank} $j$ if it has the form
\begin{equation}
w=a^{k_j}u_j a^{k_{j-1}}u_{j-1}\cdots a^{k_1} u_1
\end{equation}
where $u_i\in W_{A\setminus\{a\}\cup B}\setminus\{\emptyset\}$ and $k_i\in\Z\setminus\{0\}$, for $1\leq i\leq j$.
\end{definition}

\begin{lemma}\label{goodlemma}
Let $s\subseteq A\times\omega\times\omega$ be finite such that $s_a$ is a partial injection for all $a\in A$. Fix $a\in A$, and let $w\in W_{A\cup B}$ be $a$-good. Then for any $n\in\omega\setminus\dom(s_a)$ and $C\subseteq \omega$ finite there are cofinitely many $m\in\omega$ such that
$$
(\forall l\in\omega) e_w[s\cup\{(a,n,m)\},\rho](l)\in C\iff e_w[s,\rho](l)\halts \wedge\ e_w[s,\rho](l)\in C
$$
\end{lemma}
\begin{proof}
By induction on the rank $j$. Let $w$ be an $a$-good word of rank $1$,
$$
w=a^{k_1}u_1.
$$
Assume first $k_1>0$. Then pick $m\notin\dom(a)$ and $m\notin C$. Suppose $e_w[s\cup\{(a,n,m)\},\rho](l)\in C$ but $e_w[s,\rho](l)\undef$. Then there is some $0<i<k_1$ such that $e_{a^iu_1}[s,\rho](l)=n$. If $i<k_1-1$ then $e_{a^{i+1}u_1}[s\cup\{(a,n,m)\},\rho](l)\undef$, so we must have $i=k_1-1$. But then $e_w[s\cup\{(a,n,m)\},\rho](l)=m\notin C$, a contradiction.

Assume then $k_1<0$. Pick $m\notin\ran(e_{a^iu_1}[s,\rho])$ for all $k_1\leq i<0$. If $e_w[s\cup\{(a,n,m)\},\rho](l)\in C$ but $e_w[s,\rho](l)\undef$, then there is some $k_1<i<0$ such that $e_{a^iu_1}[s,\rho](l)\halts$ but $e_{a^{i-1}u_1}[s,\rho](l)\undef$. Since $e_{a^iu_1}[s,\rho](l)\neq m$, it follows that $e_{a^{i-1}u_1}[s\cup\{(a,n,m)\},\rho]\undef$, a contradiction.

Now let $w$ be $a$-good of rank $j>1$, and write $w=a^{k_j}u_j\bar w$, where $\bar w$ is $a$-good of rank $j-1$. Let $C'=e_{u_j^{-1}a^{-k_j}}[s,\rho](C)$. By the inductive assumption there is $I_0\subseteq\omega$ cofinite such that for all $m\in I_0$,
$$
(\forall l\in\omega) e_{\bar w}[s\cup\{(a,n,m)\},\rho](l)\in C'\iff e_{\bar w}[s,\rho](l)\halts \wedge\ e_{\bar w}[s,\rho](l)\in C'.
$$
Let $I_1\subseteq \omega$ be cofinite such that for all $m\in I_1$,
\begin{align*}
&(\forall l\in\omega) e_{a^{k_i}u_j}[s\cup\{(a,n,m)\},\rho](l)\in C\\&\iff e_{a^{k_i}u_j}[s,\rho](l)\halts \wedge\ e_{a^{k_i}u_j}[s,\rho](l)\in C.
\end{align*}
Then let $m\in I_1\cap I_0$, and suppose $e_w[s\cup\{(a,n,m)\},\rho](l)\in C$. Then $e_{\bar w}[s\cup\{(a,n,m)\},\rho](l)\in C'$ and so $e_{\bar w}[s,\rho](l)\in C'$. It follows that
$$
e_{a^{k_j}u_j}[s\cup\{(a,n,m)\},\rho](e_{\bar w}[s,\rho](l))\in C
$$
and so we have $e_{a^{k_j}u_j}[s,\rho](e_{\bar w}[s,\rho](l))=e_w[s,\rho](l)\in C$, as required.
\end{proof}

\begin{proof}[Proof of Lemma \ref{extension}]
(1) It suffices to prove this when $F=\{w\}$. Further, we may assume that $a$ occurs in $w$, since otherwise there is nothing to show.

If $w$ is $a$-good, then the statement follows from Lemma \ref{goodlemma}. If $w$ is not $a$-good, then write $w=uva^k$ (without cancellation), where $u\in W_{A\setminus\{a\}\cup B}$, $v$ is $a$-good, and $k\in\Z$. Let $\bar w=v a^k u$. Then $\bar w$ is $a$-good, and so there is $I\subseteq\omega$ cofinite such that
$$
(\forall m\in I)(s\cup\{(a,n,m)\},\{\bar w\})\leq_{\Q_{A,\rho}} (s,\{\bar w\}).
$$
We claim that $(s\cup\{(a,n,m)\},\{ w\})\leq (s,\{w\})$ when $m\in I$. Indeed, if $e_w[s\cup\{(a,n,m)\},\rho](l)=l$ then by Lemma \ref{nocancel} it follows that
$$
e_{\bar w}[s\cup\{(a,n,m)\},\rho](e_{va^k}[s\cup\{(a,n,m)\},\rho](l))=e_{va^k}[s\cup\{(a,n,m)\},\rho](l)
$$
and so
$$
e_{\bar w}[s,\rho](e_{va^k}[s\cup\{(a,n,m)\},\rho](l))=e_{va^k}[s\cup\{(a,n,m)\},\rho](l).
$$
Applying Lemma 2.2 once more, we get $e_w[s,\rho](l)=l$.

(2) Let $(s,F)\in\Q_{A,\rho}$, $a\in A$, and suppose $m_0\notin \ran(s_a)$. As above, we may assume that $F=\{w\}$. Define $\bar s\subseteq A\times\omega\times\omega$ by
$$
(x,n,m)\in \bar s\iff (x\neq a\wedge (x,n,m)\in s)\vee (x=a\wedge (x,m,n)\in s).
$$
Let $\bar w$ be the word in which every occurrence of $a$ is replaced with $a^{-1}$. Notice that $e_{\bar w}[\bar s,\rho]=e_w[s,\rho]$, and that $m_0\notin\dom(\bar s)$. By (1) above there are cofinitely many $n$ such that $(\bar s\cup\{(a,m_0,n)\},\{\bar w\})\leq (\bar s,\{\bar w\})$, and so for cofinitely many $n$ we have $(s\cup\{(a,n,m_0)\},\{w\})\leq (s,\{w\})$.
\end{proof}

The following easy consequence of Lemma \ref{extension} will be useful. We leave the proof to the reader.

\begin{corollary}\label{saturate}
Let $w\in W_{A\cup B}$, and let $A_0\subseteq A$ be the set of letters from $A$ occurring in $w$. For any condition $(s,F)\in\Q_{A,\rho}$ and finite sets $C_0,C_1\subseteq\omega$ there is $t\subseteq A_0\times\omega\times\omega$ such that $(t\cup s,F)\leq (s,F)$ and $\dom (e_w[s\cup t,\rho])\supset C_0$ and $\ran(e_w[s\cup t,\rho])\supset C_1$.
\end{corollary}

\begin{lemma}\label{forcinglemma}
Let $w\in \What{A\cup B}$ and suppose $(s,F)\forces_{\Q_{A,\rho}} e_w[\rho_G](n)=m$ for some $n,m\in\omega$. Then $e_w[s,\rho](n)\halts$ and $e_w[s,\rho](n)=m$.
\end{lemma}

\begin{proof}
By induction on the number of letters from $A$ occurring in $w$. If no letter from $A$ occurs, the statement is vacuously true. So suppose now that the above is known to hold for words with at most $k$ letters from $A$ occurring, and let $w$ be a letter with $k+1$ letters from $A$ occurring. For a contradiction, assume that $e_w[s,\rho](n)\undef$, but $(s,F)\forces_{\Q_{A,\rho}} e_w[\rho_G](n)=m$. Then we may find $a\in A$ such that $w=ua^{i}v$ without cancellation, $i\in\{-1,1\}$, and $u,v\in W_{A\cup B}$ are (possibly empty) words, such that $e_v[s,\rho](n)\halts$ but $e_v[s,\rho](n)\notin\dom (e_{a^i}[s,\rho])$. The word $w$ can be written $w=w_1w_0$ without cancellation where $w_0$ is $a$-good and $a$ does not occur in $w_1$. Note that if $e_{w_1}[s,\rho]$ is not totally defined then $\dom(e_w[s,\rho])$ is finite. By repeatedly applying Lemma \ref{goodlemma} and Lemma \ref{extension} we can find $s_1\subseteq \{a\}\times\omega\times\omega$ finite such that $s\cup s_1$ satisfies $(s\cup s_1,F)\leq (s,F)$ and such that $e_{w_0}[s\cup s_1,\rho](n)\halts$ and $n_1 = e_{w_0}[s\cup s_1,\rho](n)\neq e_{w_1}[s,\rho]^{-1}(m)$ if it is defined. Since $(s,F)\forces_{\Q_{A,\rho}} e_w[\rho_{\dot G}](n)=m$ and $(s\cup s_1,F)\forces e_{w_0}[\rho_{\dot G}](n)=n_1$ we must have $(s\cup s_1,F)\forces e_{w_1}[\rho_{\dot G}](n_1)=m$. By the inductive assumption it follows that $e_{w_1}[s\cup s_1,\rho](n_1)=m$. Since $a$ does not occur in $w_1$ it follows that $e_{w_1}[s,\rho](n_1)=m$, contradicting the choice of $n_1$.
\end{proof}

\begin{prop}\label{p.rhoG}
Let $G$ be $\Q_{A,\rho}$-generic. Then $\rho_G$, defined in \ref{rho_G}, is a function $A\cup B\to S_\infty$ such that $\rho_G\restrict B=\rho$, and $\rho_G$ induces a cofinitary representation $\hat\rho_G:\F_{A\cup B}\to S_\infty$ satisfying $\hat\rho_G\restrict\F_B=\hat\rho$.
\end{prop}

\begin{proof}
For each $a\in A$ and $n\in\omega$, let
$$
D_{a,n}=\{(s,F)\in\Q_{A,\rho}: (\exists m) (a,n,m)\in s\}
$$
and let
$$
R_{a,n}=\{(s,F)\in\Q_{A,\rho}: (\exists m) (a,m,n)\in s\}.
$$
For $w\in \What{A\cup B}$, let
$$
D_w=\{(s,F)\in\Q_{A,\rho}: w\in F\}.
$$
Then $D_w$ is easily seen to be dense, and $D_{a,n}$ and $R_{a,n}$ are dense by Lemma \ref{extension}. Thus $\rho_G$ is a function $A\cup B\to S_\infty$ as promised. It remains to prove that $\rho_G$ induces a cofinitary representation. For this let $w\in W_{A\cup B}$. Then we can find $w'\in \What{A\cup B}$ and $u\in W_{A\cup B}$ such that $w=u^{-1}w'u$. Since $D_{w'}$ is dense, there is some condition $(s,F)\in G$ such that $w'\in F$. Suppose then that $e_{w'}[\rho_G](n)=n$ in $V[G]$. Then there is some condition $(t,E)\leq_{\Q_{A,\rho}} (s,F)$ and $(t,E)\in G$ forcing this. It follows by Lemma \ref{forcinglemma} that $e_{w'}[t,\rho](n)=n$. But then by the definition of $\leq_{\Q_{A,\rho}}$ we have $e_{w'}[s,\rho](n)=n$, and so $\fix(e_{w'}[\rho_G])=\fix(e_{w'}[s,\rho])$, which is finite. Finally, $\fix(e_w[\rho_G])=e_{u}[\rho_G]^{-1}(\fix (e_{w'}[\rho_G]))$, so $\fix(e_w[\rho_G])$ is finite.
\end{proof}

Suppose $A=A_0\cup A_1$ where $A_0\cap A_1=\emptyset$ and $A_0,A_1\neq\emptyset$. We will now describe how forcing with $\Q_{A,\rho}$ over $V$ may be broken down into a two-step iteration, first forcing with $\Q_{A_0,\rho}$ over $V$, and then with $\Q_{A_1,\rho_G}$ over $V[G]$, when $G$ is $\Q_{A_0,\rho}$-generic over $V$.

\medskip

{\it Notation.} For $s\subseteq A\times \omega\times\omega$  and $A_0\subseteq A$, write $s\restrict A_0$ for $s\cap A_0\times\omega\times\omega$. For a condition $p=(s,F)\in\Q_{A,\rho}$ we will write $p\restrict A_0$ for $(s\restrict A_0, F)$, and $p\srestrict A_0$ (``strong restriction'') for $(s\restrict A_0, F\cap \What{A_0\cup B})$. (So $p\srestrict A_0$ is a condition of $\Q_{A_0,\rho}$ but $p\restrict A_0$ is in general still only a condition of $\Q_{A,\rho}$.)

For the notion of complete containment see section~\ref{complete_embeddings}.

\begin{lemma}\label{complete}
If $A_0\subseteq A$ then $\Q_{A_0,\rho}$ is completely contained in $\Q_{A,\rho}$.
\end{lemma}

\begin{proof}
Let $A_1=A\setminus A_0$. We may of course assume that $A_0,A_1\neq\emptyset$, since otherwise there is nothing to show. We first note that all $\Q_{A_0,\rho}$ conditions are also $\Q_{A,\rho}$ conditions, and so $\Q_{A_0,\rho}\subseteq \Q_{A,\rho}$. Clearly $p\leq_{\Q_{A_0,\rho}} q$ implies $p\leq_{\Q_{A,\rho}} q$. Moreover, if $p,q\in\Q_{A,\rho}$ and $p\leq_{\Q_{A,\rho}} q$ then clearly $p\srestrict A_0 \leq_{\Q_{A_0,\rho}} q\srestrict A_0$. Hence $p\perp^{\Q_{A_0,\rho}} q$ if and only if $p\perp^{\Q_{A,\rho}} q$. It remains to see that if $q\in\Q_{A,\rho}$ then there is $p_0\in\Q_{A_0,\rho}$ such that whenever $p\leq_{\Q_{A_0,\rho}} p_0$ then $p$ and $q$ are $\leq_{\Q_{A,\rho}}$-compatible. This follows from the next claim.

\begin{claim}\label{completeclaim}
For every $(s,F)\in\Q_{A,\rho}$ there is $t_0\supseteq s\restrict A_0$, $t_0\subseteq A_0\times\omega\times\omega$, such that if $(t,E)\leq_{\Q_{A_0,\rho}} (t_0,F\cap \What{A_0\cup B})$ then $(s\cup t, F)\leq_{\Q_{A,\rho}} (s,F)$. Thus, for any $q\in\Q_{A,\rho}$ there is $p_0\leq_{\Q_{A_0,\rho}} p\srestrict A_0$ such that whenever $p\leq_{\Q_{A_0,\rho}} p_0$ then $p$ is $\leq_{\Q_{A,\rho}}$-compatible with $q$.
\end{claim}

\noindent To see this, let $\{w_1,\ldots, w_n\}=F\setminus W_{A_0\cup B}$. Then each word $w_i$ may be written
$$
w_i=u_{i,k_i}v_{i,k_i}\cdots u_{i,1} v_{i,1} u_{i,0}
$$
where $u_{i,j}\in W_{A_0}$ and $v_{i,j}\in W_{A_1}$, all words are nonempty except possibly $u_{i,k_i}$ and $u_{i,0}$, and each $v_{i,j}$ starts and ends with a letter from $A_1$. By repeated applications of Corollary \ref{saturate} to $(s,F)$ and the $u_{i,j}$ we can find $t\subseteq A_0\times\omega\times\omega$ such that $t_0\supseteq s\restrict A_0$ and $\dom(e_{u_{i,j}}[s\cup t,\rho])\supseteq\ran(e_{v_{i,j}}[s,\rho])$ and $\ran(e_{u_{i,j}}[s\cup t_0,\rho]\supseteq \dom(e_{v_{i,j+1}}[s,\rho])$, and satisfying $(s\cup t_0,F)\leq_{\Q_{A,\rho}} (s,F)$. Suppose now $(t,E)\leq_{\Q_{A_0,\rho}} (t_0,F\cap \What{A_0\cup B})$. If $e_{w_i}[s\cup t,\rho](n)\halts$ for some $n\in\omega$, then by definition of $t_0$ we must have that $e_{w_i}[s\cup t_0,\rho](n)\halts$. Therefore if $e_{w_i}[s\cup t,\rho](n)=n$ we have $e_{w_i}[s\cup t_0,\rho](n)=n$, and so since $(s\cup t_0,F)\leq_{\Q_{A,\rho}} (s,F)$ it follows that $e_{w_i}[s,\rho](n)=n$. Thus $(s\cup t,F)\leq_{\Q_{A,\rho}} (s,F)$ as required.
\end{proof}

\begin{remark}\label{r.strong_embedding_mcg}
Note that in Claim~\ref{completeclaim} we obtained in fact a slightly stronger property than stated, namely the following.
Let $A\subseteq\dom(\Q)$, $p=(s,F)\in\Q$. Then there is $t_0\subseteq \oc(s)\cap A\times\omega\times\omega$ such that $s\restrict A\subseteq t_0$,
$\oc(t_0)=\oc(s)\cap A$, $(t_0, F\cap \What{A})\leq_{\Q_{\oc(p)\cap A}} p\srestrict A$ and whenever $(t,E)\leq_\Q (t_0,\F\cap \What{A})$ is such that
$\oc(t)\cap(\oc(p)\backslash A)=\oc(E)\cap (\oc(p)\backslash A)=\emptyset$, then $(t\cup s,F)\leq (s,F)$, $(t\cup s,E)\leq (t,E)$, and so
$(t\cup s, E\cup F)$ is a common extension of $(s,F)$ and $(t,E)$.
\end{remark}

\begin{lemma}\label{forcingcompatible}
Let $A=A_0 \cup A_1$. If $(t,E)\in\Q_{A_0,\rho}$ and
$$
(t,E)\forces_{\Q_{A_0,\rho}} (s_0,F_0)\leq_{\Q_{A_1,\rho_{\dot G}}} (s_1,F_1)
$$
then $(t\cup s_0,F_0)\leq_{\Q_{A,\rho}} (t\cup s_1,F_1)$.
\end{lemma}
\begin{proof}
Let $w\in F_1$ and suppose $e_w[t\cup s_0,\rho](n)=n$. If $G$ is $\Q_{A_0,\rho}$-generic such that $(t,E)\in G$ then in $V[G]$ we have $e_w[s_0,\rho_G](n)=n$, and so in $V[G]$ we have $e_w[s_1,\rho_G](n)=n$, from which it follows that $e_w[t\cup s_1,\rho](n)=n$.
\end{proof}

\begin{lemma}\label{twostep}
Suppose $G$ is $\Q_{A,\rho}$-generic over $V$, and $A=A_0\cup A_1$ where $A_0,A_1\neq\emptyset$, $A_0\cap A_1=\emptyset$. Then $H=G\cap\Q_{A_0,\rho}$ is $\Q_{A_0,\rho}$-generic over $V$ and
$$
K=\{p\restrict A_1: p\in G\}=\{(s\restrict A_1,F): (s,F)\in G\}
$$
is $\Q_{A_1,\rho_H}$-generic over $V[H]$. Moreover, $\rho_G=(\rho_H)_K$.
\end{lemma}
\begin{proof}
That $H$ is $\Q_{A_0,\rho}$-generic over $V$ follows from the previous Lemma. To see that $K$ is $\Q_{A_1,\rho_H}$-generic over $V[H]$, suppose $D\subseteq \Q_{A_1,\rho_H}$ is dense and $D\in V[H]$. Define
$$
D'=\{p\in\Q_{A,\rho}: p\srestrict A_0\forces_{\Q_{A_0,\rho}} p\restrict A_1\in \dot{D}\}
$$
and let $p_0\in H$ be a condition such that $p_0\forces_{\Q_{A_0,\rho}} \text{``$D$ is dense''}$. We claim that $D'$ is dense below $p_0$ (in $\Q_{A,\rho}$.) For this, let $(s,F)=p\leq_{\Q_{A,\rho}} p_0$. Then by Claim \ref{completeclaim} we can find $p_0\leq_{
\Q_{A_0,\rho}} p\srestrict A_0$ such that for any $p_1\leq_{\Q_{A_0,\rho}} p_0$, $p_1$ is compatible with $p$. Thus we can find $q=(s_0,F_0)\in\Q_{A_1,\rho_{H}}$ and $(t,E)\leq_{\Q_{A_0,\rho}} p_0$ such that
$$
(t,E)\forces_{\Q_{A_0,\rho}} \dot q\in \dot D\wedge \dot q\leq_{\Q_{A_1,\rho_{\dot H}}} \dot p\restrict A_1.
$$
By Lemma \ref{forcingcompatible} it holds that $(s_0\cup t,F_0)\leq_{\Q_{A,\rho}} (s\restrict A_1\cup t,F)$, and therefore
$$
(s_0\cup t,F_0\cup E)\leq_{\Q_{A,\rho}} (s,F).
$$
Since clearly $(s_0\cup t, F_0\cup E)\in D'$, this shows that $D'$ is dense below $p_0$.

Now, since $p_0\in G$ it follows that there is $q'\in D'\cap G$. In $V[H]$ it then holds that $q'\restrict A_1\in D$, which shows that $K\cap D\neq\emptyset$.

That $(\rho_H)_K=\rho_G$ follows directly from the definition of $H$ and $K$.
\end{proof}

Our next goal is to prove the following.

\begin{theorem}\label{maxlthm}
Suppose $\rho:B\to S_\infty$ induces a cofinitary representation of $\F_B$. If $\card(A)>\aleph_0$ and $G$ is $\Q_{A,\rho}$-generic over $V$, then $\im(\rho_G)$ is a maximal cofinitary group in $V[G]$.
\end{theorem}

The Theorem is a consequence of the following Lemma, which is parallel to \cite[Lemma 3.3]{zhang1}.

\begin{lemma}\label{cofinlemma}
Suppose $\rho:B\to S_\infty$ induces a cofinitary representation $\hat\rho:\F_B\to S_\infty$ and that there is $b_0\in B$ such that $\rho(b_0)\neq I$. Let $(s,F)\in\Q_{A,\rho{\:\restrict\,} B\setminus\{b_0\}}$ and let $a_0\in A$. Then there is $N\in\omega$ such that for all $n\geq N$
$$
(s\cup\{(a_0,n,\rho(b_0)(n))\},F)\leq_{\Q_{A,\rho{\:\restrict\,} B\setminus\{b_0\}}} (s,F).
$$
\end{lemma}

\begin{proof}
Let $w_1,\ldots, w_l\in F$ enumerate the words in $F$ in which $a_0$ occur. Then we may write each word $w_i$ on the form
$$
w_i=u_{i,j_i}a_0^{k(i,j_i)}u_{i,j_i-1}a_0^{k(i, j_i-1)}\cdots u_{i,1}a_0^{k(i, 1)}u_{i,0}
$$
where $u_{i,m}\in W_{A\setminus\{a_0\}\cup B\setminus\{b_0\}}$ and are non-$\emptyset$ whenever $m\notin\{j_i,0\}$. By Lemma \ref{extension} we may assume that for all $u_{i,m}$ with $\dom(e_{u_{i,m}}[s,\rho])$ and $\ran(e_{u_{i,m}}[s,\rho])$ finite that
$$
\dom(e_{a_0^{k(i,m+1)}}[s,\rho])\supseteq \ran(e_{u_{i,m}}[s,\rho])
$$
and
$$
\ran(e_{a_0^{k(i,m)}}[s,\rho])\supseteq \dom(e_{u_{i,m}}[s,\rho]).
$$
Let $\bar w_i$ be the word in which every occurrence of $a_0$ in $w_i$ has been replaced by $b_0$. If $e_{\bar w_i}[\rho]$ is totally defined, then since $\rho$ induces a cofinitary representation there are at most finitely many $n$ such that $e_{\bar w_i}[\rho](n)\neq n$. For each $\bar w_i$ with $e_{\bar w_i}[\rho]$ totally defined and $1\leq m\leq j_i$ let
$$
\bar w_{i,m}=u_{i,m}b_0^{k(i,m)}\cdots u_{i,1}b_0^{k(i,1)}u_{i,0},
$$
and let
\begin{align*}
N_i=\max\{ e_v[\rho](k):& e_{\bar w_i}[\rho](k)=k \wedge v=b^{\sign(k(i,m)p}\bar w_{i,m}\wedge
\\
&0\leq p\leq\sign(k(i,m))k(i,m)\wedge 0\leq m\leq j_i\}.
\end{align*}
Then let $N\in\omega$ be such that $N\geq\max\{N_i:i\leq l\}$ and $n\notin\dom(s_{a_0})$ and $\rho(b_0)(n)\notin\ran(s_{a_0})$ whenever $n\geq N$. Then for any $n\geq N$ we have that on the one hand, if $e_{\bar w_i}[\rho]$ is not everywhere defined, then
$$
\dom(e_{w_i}[s,\rho])=\dom(e_{w_i}[s\cup\{(a_0,n,\rho(b_0)(n))\},\rho]),
$$
while if $e_{\bar w_i}[\rho]$ is everywhere defined then necessarily
$$
e_{w_i}[s\cup\{(a_0,n,\rho(b_0)(n))\},\rho](k)=k
$$
only when $e_{w_i}[s,\rho](k)=k$.
\end{proof}

\begin{proof}[Proof of Theorem \ref{maxlthm}]
Let $b_0\notin B\cup A$. Suppose $\card(A)>\aleph_0$ and that $G$ is $\Q_{A,\rho}$-generic, and suppose further that there is a permutation $\sigma\in \cofin(S_\infty)^{V[G]}$ such that $\rho_G':B\cup\{b_0\}\to S_\infty$ defined by $\rho_G'\restrict B=\rho_G$, and $\rho_G'(b_0)=\sigma$ induces a cofinitary representation of $\F_{B\cup\{b_0\}}$. Let $\dot\sigma$ be a name for $\sigma$. Then there is $A_0\subseteq A$ countable so that $\dot\sigma$ is a $\Q_{A_0,\rho}$-name and so we already have $\sigma\in V[H]$, where $H=G\cap \Q_{A_0,\rho}$. Let $A_1=A\setminus A_0$, and let $K$ be defined as in Lemma \ref{twostep}. Define
$$
D_{\sigma, N}=\{(s,F)\in\Q_{A_1,\rho_{H}}: (\exists n\geq N) s(n)=\sigma(n)\}.
$$
By Lemma \ref{cofinlemma} this set is dense. Thus in $V[H][K]$, for any $a_0\in A\setminus A_0$ we have $(\rho_{H})_K(a_0)(n)=\sigma(n)$ for infinitely many $n$. Since $(\rho_H)_K=\rho_G$ by Lemma \ref{twostep}, this is contradicts that $\rho_G'$ induces a cofinitary representation.
\end{proof}

\section{Iteration along a two-sided template}

\subsection{Preliminaries} We now recall various definitions and introduce several notions that are needed to set up the framework in which we will treat the iteration along a two-sided template.

\subsubsection{Localization}

As indicated we are aiming to give an iterated forcing construction which will provide a generic extension in which the minimal size of a
maximal cofinitary group is of countable cofinality. In order to provide a lower bound for $\mathfrak{a}_g$, along this iteration construction cofinally often
we will force with the following partial order $\BbL$, known as localization.

\begin{definition}\label{d.loc} The forcing notion $\BbL$ consists of pairs $(\sigma,\phi)$ such that $\sigma\in{^{<\omega}(^{<\omega}[\omega])}$, $\phi\in{^\omega(^{<\omega}[\omega])}$ such that $\sigma\subseteq\phi$, $\forall i<|\sigma|(|\sigma(i)|=i)$ and for all $i\in\omega(|\phi(i)|\leq |\sigma|)$.
The extension relation is defined as follows: $(\sigma,\phi)\leq(\tau,\psi)$ if and only if $\sigma$ end-extends $\tau$ and for all $i\in\omega$ $(\psi(i)\subseteq\phi(i))$.
\end{definition}

Recall that a slalom is a function $\phi:\omega\to [\omega]^{<\omega}$ such that for all $n\in\omega$ we have $|\phi(n)|\leq n$. We say that a slalom localizes a real $f\in{^\omega\omega}$ if there is $m\in\omega$ such that for all $n\geq m$ we have $f(n)\in\phi(n)$.  The following is well-known and follows easily from the definition of $\BbL$.

\begin{lemma}\label{l.localize}
The poset $\BbL$ adds a slalom which localizes all ground model reals.
\end{lemma}

Let $\add(\mathcal N)$ denote the additivity of the (Lebesgue) null ideal, and let $\cof(\mathcal N)$ denote the cofinality of the null ideal. Then:

\begin{theorem}[Bartoszyn\'{n}ski, Judah{\cite[Ch.2]{judabarto}}]\label{t.barto}
(1) $\add(\mathcal N)$ is the least cardinality of a family $F\subseteq\omega^\omega$ such that no slalom localizes all members of $F$

(2) $\cof(\mathcal N)$ is the least cardinality of a family $\Phi$ of slaloms such that every member of $\omega^\omega$ is localized by some $\phi\in\Phi$.
\end{theorem}

Finally, we will need the following result due to Brendle, Spinas and Zhang:

\begin{theorem}[\cite{brspzh00}]\label{t.agM}
$\mathfrak a_g\geq\non(\mathcal M)$.
\end{theorem}

In our intended forcing construction cofinally often we will force with the partial order $\BbL$, which using the above
characterizations will provide a lower bound for $\mathfrak{a}_g$.

\subsubsection{Complete embeddings.}\label{complete_embeddings}
Recall that if $\P$ and $\Q$ are posets such that $\P\subseteq \Q$, then we say that $\P$ is completely contained in $\Q$, written $\P\lessdot\Q$ if $\P\subseteq\Q$ and
\begin{enumerate}
\item if $p,p'\in\P$ and $p\leq_\P p'$ then $p\leq_{\Q} p'$.
\item if $p, p'\in\P$ and $p\perp_{\P} p'$ then $p\perp_{\Q} p'$.
\item\label{i.reduction} if $q\in\Q$ then there is $r\in\P$ (called a \emph{reduction} of $q$) such that for all $p\in\P$ with $p\leq_{\P} r$, the conditions $p$ and $q$ are compatible.
\end{enumerate}
We note that (\ref{i.reduction}) above may be seen to be equivalent to
\begin{enumerate}[\indent(3')]
\item All maximal antichains in $\P$ are maximal in $\Q$.
\end{enumerate}

\begin{lemma}\label{l.strongsuslin} Let $\P$ and $\Q$ be posets, and suppose $\P\lessdot\Q$. Let $q\in\Q$, $p\in\P$ and $q\leq_\Q p$. Then any reduction of $q$ to $\P$ is compatible in $\P$ with $p$, and so $q$ has a reduction extending $p$.
\end{lemma}
\begin{proof} Suppose $r\in\P$ is a reduction of $q$ and $r\bot_\P p$. Let $x\in\P$, $x\leq_\P r$. Then since $r$ is a reduction of $q$, we have that
$x$ is compatible with $q$ in $\Q$ and so there is $x^\prime\in\Q$ which is their common extension. But then $x^\prime\leq_\Q x\leq_\P r$ and so $x^\prime\leq_\Q r$. Also $x^\prime\leq_\Q q\leq_\Q p$ and so $x^\prime\leq_\Q p$. Therefore $r$ is compatible with $p$ in $\Q$. But by assumption $\P\lessdot \Q$ and so for all $x,y\in\P(y\bot_\P z\rightarrow y\bot_\Q z)$. Therefore $r\bot_\Q p$, which is a contradiction.

To complete the proof, consider any reduction $r$ of $q$ to $\P$. Then $r$ is compatible in $\P$ with $´p$ and so they have a common extension $r_0$. However, any extension of a reduction is a reduction and so $r_0$ is a reduction of $q$ with $r_0\leq_\P p$.
\end{proof}

\subsubsection{Canonical Projection of a Name for a Real}

\begin{definition}\label{d.projection_name_real} Let $\mathbb{B}$ be a partial order and $y\in\mathbb{B}$. For each $n\geq 1$ let $\mathcal{B}_n$ be a maximal antichain below $y$. We will say that the set $\{(b,s(b))\}_{b\in\mathcal{B}_n,n\geq 1}$   is {\it{a nice name for a real below $y$}} if
\begin{enumerate}
{\item whenever $n\geq 1$, $b\in\mathcal{B}_n$ then $s(b)\in{^n\omega}$}
{\item whenever $m>n\geq 1$, $b\in \mathcal{B}_n$, $b^\prime\in \mathcal{B}_m$ and $b,b^\prime$ are compatible, then $s(b)$ is an initial segment of $s(b^\prime)$.}
\end{enumerate}
\end{definition}

\begin{remark} Whenever $\dot{f}$ is a $\mathbb{B}$-name for a real, we can associate with $\dot{f}$ a family of maximal antichains $\{\mathcal{B}_n\}_{n\geq 1}$ and initial approximations $s(b)\in{^n\omega}$ of $\dot{f}$ for  $b\in\mathcal{B}_n$ such that for all $n$ and $b$, $b\Vdash_{\mathbb{B}} \dot{f}\rest n =\check{s}(b)$ and the collection $\{(b,s(b))\}_{b\in\mathcal{B}_n,n\in\omega}$ has the above properties. Thus we can assume that all names for reals are nice and abusing notation we will write $\dot{f}=\{(b,s(b))\}_{b\in\mathcal{B}_n,n\in\omega}$.
\end{remark}

\begin{lemma}\label{l.projection_name_real} Let $\mathbb{A}$ be a complete suborder of $\mathbb{B}$, $y\in \mathbb{B}$ and $x$ a reduction of $y$ to $\mathbb{A}$. Let $\dot{f}=\{(b,s(b))\}_{b\in\mathcal{B}_n,n\geq 1}$ be a nice name for a real below $y$. Then there is $\dot{g}=\{(a,s(a))\}_{a\in\mathcal{A}_n, n\geq 1}$, a $\mathbb{A}$-nice name for a real below $x$, such that for all $n\geq 1$, for all $a\in\mathcal{A}_n$, there is
$b\in\mathcal{B}_n$ such that $a$ is a reduction of $b$ and $s(a)=s(b)$.
\end{lemma}

\begin{remark} Whenever $\dot{f},\dot{g}$ are as above, we will say that $\dot{g}$ is a {\emph{canonical projection}} of $\dot{f}$ below $x$.
\end{remark}
\begin{proof} Recursively we will construct the antichains $\mathcal{A}_n$. Along this construction we will guarantee that for all $a\in\mathcal{A}_n$,
$a^\prime\in\mathcal{A}_{n+1}$either $a^\prime\leq a$ or $a\bot a^\prime$, and that if $a^\prime\leq a$, then $s(a^\prime)$ end-extends $s(a)$.

First we will define $\mathcal{A}_1$. Let $t\in\mathbb{A}$ be an arbitrary extension of $x$. Since $x$ is a reduction of $y$, there is $\hat t\in \mathbb{B}$ such that $\hat{t}\leq_{\mathbb{B}} t,y$. Therefore there is $b\in\mathcal{B}_1$ such that $\hat{t}$ and $b$ are compatible with a common extension $\bar{t}$.
Then in particular $\bar{t}\leq_{\mathbb{B}} t$ and so we can find a reduction $a$ of $\bar{t}$ extending $t$. Since $\bar{t}\leq b$, $a$ is also a reduction of $b$. Define $s(a)=s(b)$, $a(t)=a$. Let $\mathcal{A}_1$ be a maximal antichain in the dense below $x$ set $D_1=\{a(t):t\leq x\}$.

Suppose $\mathcal{A}_n$ has been defined. Let $a\in\mathcal{A}_n$ and $t\leq_{\mathbb{A}} a$. By the inductive hypothesis, there is $b\in\mathcal{B}_n$ such that $a$ is a reduction of $b$ and $s(a)=s(b)$. Then $t$ is compatible in  $\mathbb{B}$ with $b$ with common extension $\hat{t}$. Then in particular $\hat{t}\leq_{\mathbb{B}} y$ and so there is $\bar{b}\in\mathcal{B}_{n+1}$ such that $\hat{t}$ is compatible with $\bar{b}$ in $\mathbb{B}$ with common extension $\tilde{t}$. Then in particular $\tilde{t}\leq \bar{b}, b$ and so $s(b)$ is an initial segment of $s(\bar{b})$. Since $\tilde{t}\leq t$, it has a reduction $\bar{a}\leq_\mathbb{A} t$. Define $a(t)=\bar{a}$, $s(\bar{a})=s(\bar{b})$. Again since $\tilde{t}\leq\bar{b}$, $\bar{a}$ is also a reduction of $\bar{b}$.
Let $\mathcal{A}_{n+1,a}$ be a maximal antichain in the dense below $a$ set $\{\bar{a}(t):t\leq_{\mathbb{A}}a\}$ and let $\mathcal{A}_{n+1}=\bigcup_{a\in\mathcal{A}_n} \mathcal{A}_{n+1,a}$.
\end{proof}

\subsubsection{Canonical Projection of a Name for a Slalom}

\begin{definition}\label{d.projection_name_slalom} Let $\mathbb{B}$ be a partial order and $y\in\mathbb{B}$. Let  $\sigma\in{^{<\omega}(^{<\omega}[\omega])}$ be such that $\forall i<|\sigma|(|\sigma(i)|=i)$, and for each $n\geq 1$ let $\mathcal{B}_n$ be a maximal antichain below $y$. We will say that the pair $(\check{\sigma},\dot{\phi})$ is {\it{a nice name for an
element of $\mathbb{L}$ below $y$}}, where $\dot{\phi}=\{(b,\sigma(b))\}_{b\in\mathcal{B}_n,n\geq 1}$, if the following conditions hold:
\begin{enumerate}

{\item whenever $n\geq 1$ and $b\in\mathcal{B}_n$ then $\sigma(b)\in{^n(^{<\omega}[\omega])}$}
{\item whenever $1\leq n\leq|\sigma|$ and $b\in\mathcal{B}_n$ then  $\sigma(b)=\sigma\rest n$}
{\item whenever $n > |\sigma|$, then $\sigma\subset\sigma(b)$ and $\forall i: |\sigma|\leq i<n(|\sigma(b)(i)|\leq |\sigma|)$,}
{\item whenever $m>n\geq|\sigma|$, $b\in \mathcal{B}_n$, $b^\prime\in \mathcal{B}_m$ and $b,b^\prime$ are compatible, then $\sigma(b)$ is an initial segment of $\sigma(b^\prime)$.}
\end{enumerate}
\end{definition}

\begin{remark} If $(\check{\sigma},\dot{\phi})$ where $\dot{\phi}=\{(b,\sigma(b))\}_{b\in\mathcal{B}_n,n\geq 1}$ is a nice name for an element of $\mathbb{L}$ below $y$, then $y\Vdash (\check{\sigma},\dot{\phi})\in\mathbb{L}$ and for all $n\in\omega$, $b\in\mathcal{B}_n$ $b\Vdash\dot{\phi}\restrict n= \check{s(b)}$.
\end{remark}

\begin{lemma}\label{l.projection_name_slalom} Let $\mathbb{A}$ be a complete suborder of $\mathbb{B}$, $y\in \mathbb{B}$ and $x$ a projection of $y$ to $\mathbb{A}$. Let $(\check{\sigma},\dot{\phi})$ where $\dot{\phi}=\{(b,\sigma(b))\}_{b\in\mathcal{B}_n,n\geq 1}$ be a nice name for an element of $\mathbb{L}$ below $y$. Then
there is an $\mathbb{A}$-nice name $(\check{\sigma},\dot{\psi})$ where $\dot{\psi}=\{(a,\sigma(a))\}_{a\in\mathcal{A}_n, n\geq 1}$ for an element in $\mathbb{L}$ below $x$ such that for all $n\geq 1$, for all $a\in\mathcal{A}_n$, there is $b\in\mathcal{B}_n$ such that $a$ is a reduction of $b$ and $\sigma(a)=\sigma(b)$.
\end{lemma}
\begin{proof} Similar to the proof of~\ref{l.projection_name_real}.
\end{proof}

Another forcing notion which will be of interest for us is Hechler forcing $\mathbb{H}$. Recall that it consists of pairs $(s,f)\in{^{<\omega}\omega\times{^\omega\omega}}$ such that $s\subseteq f$ and extension relation $(s,f)\leq (t,g)$ iff $s$ end-extends $t$ and for all $i\in\omega(g(i)\leq f(i))$. Clearly, if $y$ forces that $(\check{s},\dot{f})$ is a condition in $\mathbb{H}$
and $\dot{f}$ is a nice name for a real below $y$, then $\dot{f}$ has a canonical projection $\dot{f}^\prime$ below $x$ such that $x$ forces that $(\check{s},\dot{f}^\prime)$ is a Hechler condition.

\subsubsection{Suslin, $\sigma$-Suslin and good $\sigma$-Suslin posets}

Recall that a \emph{Suslin poset} is a poset $(\BbS,\leq_{\BbS})$ such that $\BbS$$(\subseteq\omega^\omega)$, $\leq_{\BbS}$ and $\perp_{\BbS}$ have ${\mathbf\Sigma}^1_1$ definitions (with parameters in the ground model.) For a Suslin forcing $\BbS$, the ordering $\leq_{\BbS}$ will be defined by the $\mathbf\Sigma^1_1$ predicate in whatever model we work in (that has a code for $\leq_{\BbS}$.) The key property of Suslin forcings that we need is the following well-known fact. 

\begin{lemma}\label{l.suslinembedd}
Let $\P$ and $\Q$ be posets and let $\BbS$ be a c.c.c. Suslin poset. If $\P\lessdot\Q$ then $\P*\dot{\BbS}\lessdot \Q*\dot{\BbS}$ (where $\dot{\BbS}$ denotes the name of $\BbS$ for the relevant poset.)
\end{lemma}

We will work with the following strengthening of the notion of Suslin forcing:

\begin{definition}\label{d.sigma_suslin}
Let $(\BbS,\leq_{\BbS})$ be a Suslin forcing notion, whose conditions can be written in the form $(s,f)$ where $s\in{^{<\omega}\omega}$ and $f\in{^\omega\omega}$. We will say that $\BbS$ is {\emph{$n$-Suslin}}  if whenever $(s,f)\leq_{\BbS} (t,g)$ and $(t,h)$ is a condition in $\BbS$ such that $h\rest {n\cdot |s|}= g\rest {n\cdot |s|}$ then $(s,f)$ and $(t,h)$ are compatible. A forcing notion is called {\emph{$\sigma$-Suslin}} if it is
$n$-Suslin for some $n$.
\end{definition}

Clearly, if $\BbS$ is $n$-Suslin and $m\geq n$, then $\BbS$ is also $m$-Suslin. If $\BbS$ is $n$-Suslin and $(s,f)$ and $(s,g)$ are conditions in $\BbS$ such that $f\rest {n\cdot |s|}=g\rest {n\cdot|s|}$ then $(s,f)$ and $(s,g)$ are compatible. Thus every $\sigma$-Suslin forcing notion is $\sigma$-linked and so has the Knaster property. Hechler forcing $\mathbb{H}$ is $1$-Suslin, localization $\mathbb{L}$ is $2$-Suslin.

\begin{definition}\label{d.good_suslin}Let $(\BbS,\leq_{\BbS})$ be a Suslin forcing notion, whose conditions can be written in the form $(s,f)$ where $s\in^{<\omega}\omega$, $f\in{^\omega\omega}$.
\begin{enumerate}
{\item The pair $(\check{s},\dot{f})$ is {\emph{a nice name for a condition}} in $\BbS$ below $y\in\mathbb{B}$ if $\dot{f}$ is a nice name for a real below $y$ and $y\Vdash_{\mathbb{B}}(\check{s},\dot{f})\in\dot{\BbS}$. }
{\item Whenever $(\check{s},\dot{f})$ is a nice name for a condition in $\BbS$ below $y\in\mathbb{B}$, $x\in\mathbb{A}$ is a reduction of $y$ and $\dot{g}$
is a canonical projection of $\dot{f}$ below $x$ such that $x\Vdash_{\mathbb{B}}(\check{s},\dot{g})\in\dot{\BbS}$, we will say that $(\check{s},\dot{g})$ is {\emph{a canonical projection of the nice name}} $(\check{s},\dot{f})$ below $x$.}
{\item $\BbS$ is called {\emph{good}} if every nice name for a condition in $\BbS$ below $y$ has a canonical projection below $x$, whenever $x\in\mathbb{A}$ is a reduction of $y\in\mathbb{B}$.}
\end{enumerate}
\end{definition}

As an immediate corollary of Lemma~\ref{l.projection_name_slalom} we obtain that the localization poset $\mathbb{L}$ is a good $\sigma$-Suslin forcing notion.
It is straightforward to verify that the Hechler poset $\mathbb{H}$ is good $\sigma$-Suslin.

\subsubsection{Finite function posets}

\begin{definition}\label{finite_function_poset}  Let $A$ be fixed sets and let $\Q$ be a poset of pairs $p=(s^p,F^p)$ where
$s^p\subseteq A\times\omega\times\omega$ is finite, for every $a\in A$, $s^p_a=\{(n,m): (a,n,m)\in s\}$ is a finite partial function and
$F\in [\What{A}]^{<\omega}$. For $p\in\Q$ let $\oc(s^p)=\{a: \exists n,m(a,n,m)\in s^p\}$ and let
$\hbox{oc}(p)=\oc(s^p)\cup\{a:a\;\hbox{is a letter from a word in}\; F^p\}$. For $B\subseteq A$ let $p\rest B=(s^p\cap B\times\omega\times\omega, F^p)$, let $p\srestrict B=(s^p\cap B\times\omega\times\omega, F^p\cap \What{B})$ and let $\hbox{dom}(\Q)=A$. Then $\Q$ is a {\emph{finite function poset (with side conditions)}} if:
\begin{enumerate}[\indent (i)]
\item "Restrictions" whenever $p,q\in \Q$, $B\subseteq A$ then
 \begin{itemize}
 \item $p\restrict B$, $p\srestrict B$ are conditions in $B$, and $p\restrict B\leq p\srestrict B$,
 \item if $p\leq q$ then $p\srestrict B\leq q\srestrict B$.
 \end{itemize}
\item "Extensions" whenever $p=(s,F)\in\Q$
 \begin{itemize}
 \item and $t\subseteq A\times\omega\times\omega$ is finite such that $\oc(p)\cap\oc(t)=\emptyset$, then $(s\cup t,F)\leq p$;
 \item and $E\in [\What{A}]^{<\omega}$ contains $F$, then $(s,E)\leq (s,F)$.
 \end{itemize}
\end{enumerate}
Whenever $B\subseteq\dom(\Q)$ by $\Q_B$ we denote the suborder $\{p\srestrict B: p\in\Q\}$.
\end{definition}

\begin{definition}\label{d.strong_reduction}
Let $\Q$ be a finite function poset. We say that $\Q$ has {\emph{the strong embedding property}} if whenever $A_0\subseteq \dom(\Q)$,
and $p=(s,F)\in\Q$, then there is $t_0\subseteq (oc(s)\cap A_0)\times\omega\times\omega$ such that $s\restrict A_0\subseteq t_0$,
$(t_0, F\cap \What{A_0})\leq_{\Q_{oc(p)\cap A_0}} p\srestrict A_0$ and whenever $(t,E)\leq_\Q (t_0, F\cap\What{A_0})$ is such that
$\oc(t)$ and $\oc(E)$ are disjoint from $\oc(p)\backslash A$, then $(t\cup s,F)\leq (s,F)$ and $(t\cup s,E)\leq (t,E)$. We say that
$(t_0, F\cap\What{A_0})$ is a {\emph{strong reduction}} of $p$ and $(s\cup t,F\cup F)$ a {\emph{canonical extension}} of $(s,F)$ and $(t,E)$.
\end{definition}

\begin{remark}\label{d.ffe}
Note that if $\Q$ is a finite function poset with the strong embedding property then whenever $A\subseteq B\subseteq\dom(\Q)$, $C\subseteq\dom(\Q)$ are
such that $C\cap B=A$, for every condition $p\in\Q\restrict B$ there is $p_0\leq_{\Q\restrict A} p\restrict A$ such that $\oc(p_0)=\oc(p)\cap A$ and if $q_0$ is a $\Q\restrict C$-extension of $p_0$, then $q_0$ is compatible with $p$. We will say that $p_0$ is a strong $\Q\rest A$-reduction of $p$.
\end{remark}

\begin{lemma}
$\Q_{A,\rho}$ is a finite function poset with the strong embedding property.
\end{lemma}

Another example of a finite function poset with the strong embedding property is the following forcing notion $\D_A$. Let $A$ be a nonempty set and let
$\D_A$ be the poset of all pairs $(s^p, F^p)$ where $s^p\subseteq A\times\omega\times 2$ is a finite set such that for all $a\in A$, $s^p_a=\{(n,m): (a,n,m)\in s\}$ is a finite partial function and $F\in [A]^{<\omega}$.  The condition $q$ is said to extend $p$ iff $s^q\supset s^p$, $F^q\supset F^p$ and for all $a,b\in F^p$ we have that $s^a_q\cap s^b_q\subseteq s^a_p\cap s^b_p$. If $|A|>\omega$, then $\D_A$ adds a maximal almost disjoint family of size $|A|$.

\subsection{Two-sided templates} If $(L,\leq)$ is a linearly ordered set and $x\in L$, we let $L_x=\{y\in L:y<x\}$ and $L_x^==\{y\in L: y\leq x \}$. If $L_0\subseteq L$ is a distinguished subset of $L$ and $A\subseteq L$, then the \emph{$L_0$-closure} of $A$ is defined as
$$
\cl_{L_0}(A)=A\cup\bigcup_{x\in A} L_x\cap L_0,
$$
and we will say that $A$ is \emph{$L_0$-closed} if $A=\cl_{L_0}(A)$. Note that $\cl_{L_0}(A)$ is the smallest set $B\supseteq A$ with the property that if $x\in B$ then $L_x\cap L_0\subseteq B$. We will usually drop mention of $L_0$ when it is clear from the context, and write ``closed'' instead of ``$L_0$-closed'' and write $\cl$ instead of $\cl_{L_0}$.

\begin{definition}[J. Brenlde,~\cite{brendle03}]\label{d.template}
A \emph{two-sided template} is a $4$-tuple $\mathcal T=((L,\leq),\mathcal I, L_0, L_1)$ consisting of a linear ordering $(L,\leq)$, a family $\mathcal I\subseteq\mathcal P(L)$, and a decomposition $L=L_0\cup L_1$ into two disjoint pieces such that the following holds:
\begin{enumerate}
\item $\mathcal I$ is closed under finite intersections and unions, and $\emptyset,L\in\mathcal I$.
\item If $x,y\in L$, $y\in L_1$ and $x<y$ then there is $A\in\mathcal I$ such that $A\subseteq L_y$ and $x\in A$.
\item If $A\in\mathcal I$, $x\in L_1\backslash A$, then $A\cap L_x\in \mathcal I$.
\item The family $\{A\cap L_1: A\in \mathcal I\}$ is well-founded when ordered by inclusion.
\item All $A\in\mathcal I$ are $L_0$-closed.
\end{enumerate}
Given a two-sided template $\mathcal T$ as above, $x\in L$ and $A\in\mathcal I$, we define
$$
\mathcal I_A=\{B\in \mathcal I: B\subset A\},
$$
$$
\mathcal I_x=\{B\in\mathcal I: B\subseteq L_x\}
$$
and $\mathcal I_{A,x}=\mathcal I_A\cap\mathcal I_x$. Finally we define the rank function $\Dp:\mathcal I\to\on$ by letting $\Dp(A)=0$ for $A\subseteq L_0$ and $\Dp(A)=\sup\{\Dp(B)+1: B\in\mathcal I\wedge B\cap L_1\subset A\cap L_1\}$. We define $\Rk(\mathcal T)$, the \emph{rank} of $\mathcal T$, to be $\Rk(\mathcal T)=\Dp(L)$.

If $A\subseteq L$ then $\mathcal T_A$ is the template $((A,\leq), \mathcal I\restrict A, L_0\cap A,L_1\cap A)$, where
$$
\mathcal I\restrict A=\{A\cap B:B\in \mathcal I\}.
$$
Note that if $A\in\mathcal I$ then $\Rk(\mathcal T_A)=\Dp(A)$. Moreover, if $A\subseteq L$ is arbitrary, then $\Rk(\mathcal T_A)\leq\Rk(\mathcal T)$.
\end{definition}

\subsection{Iteration along a two-sided template}
We are now ready to define the iteration along a two-sided template. This definition is a generalization of the definition of iterating ''Hechler forcing and adding a mad family along a template" given in ~\cite{brendle03}.

\begin{definition}\label{template_poset}
Let $\mathcal T=((L,\leq),\mathcal I,L_0,L_1)$ be a two-sided template, $\Q$ a finite function forcing with the strong embedding property such that $L_0=\dom(\Q)$ and $\BbS$ a good $\sigma$-Suslin forcing notion. The poset $\P(\mathcal T,\Q,\BbS)$ is defined recursively according the following clauses:
\begin{enumerate}
\item If $\Rk(\mathcal T)=0$, then $\P(\mathcal T,\Q,\BbS)=\Q_{L_0}$.
\item Assume that for all $\mathcal T$ with $\Rk(\mathcal T)<\kappa$, $\P(\mathcal T,\Q,\BbS)$ has been defined (and {\emph{is}} a poset, see comment below). Let $\mathcal T$ be a two-sided template of rank $\kappa$, and for $B\in\mathcal I$ of $\Dp(B)<\kappa$ let $\P_B=\P(\mathcal T_B,\Q,\BbS)$. We define $\P=\P(\mathcal T,\Q,\BbS)$ as follows:

\begin{enumerate}[(i)]
\item $\P$ consists of all pairs $P=(p,F^p)$ where $p$ is a finite partial functions with $\dom(p)\subseteq L$, $P\restrict_{L_0}:=(p\restrict L_0, F^p)\in\Q$ and if
$x_p\overset{\rm def} = \max\{\dom(p)\cap L_1\}$ is defined then there is $B\in\mathcal I_{x_p}$ (called a \emph{witness} that $P\in\P$) such that $P\srestrict L_{x_p}:=(p\restrict L_{x_p},F^p\cap\What{B})\in\P_B$, $p(x_p)=(\check{s}_x^p,\dot{f}_x^p)$, where $s^p_x\in{^{<\omega}\omega}$, $\dot{f}^p_x$ is a  $\P_B$ name for a real and $(P\srestrict L_{x_p},p(x_p))\in \P_B*\dot\BbS$.
\item For $P,Q\in\P$, let $Q\leq_{\P}P$ iff $\dom(p)\subseteq\dom(q)$, $(q\restrict L_0, F^q)\leq_{\Q} (p\restrict L_0,F^p)$, and if $x_p$ is defined then either
\begin{itemize}
\item[(ii.a)] $x_p<x_q$ and $\exists B\in\mathcal I_{x_q}$ such that $P\srestrict L_{x_q}, Q\srestrict L_{x_q}\in\P_B$ and $Q\srestrict L_{x_q}\leq_{\P_B} P\srestrict L_{x_q}$,
\end{itemize}
\noindent or
\begin{itemize}
\item[(ii.b)] $x_p=x_q$ and $\exists B\in\mathcal I_{x_q}$ witnessing $P,Q\in\P$, and such that  $$(Q\srestrict L_{x_q}, q(x_q))\leq_{\P_B*\dot\BbS} (P\srestrict L_{x_p}, p(x_p)).$$
\end{itemize}
\end{enumerate}
\end{enumerate}
Below we will call $B$ as in (ii.a) or (ii.b) a \emph{witness} to $Q\leq_{\P} P$.
\end{definition}

Whenever the side condition $F^p$ is clear from the context, we will denote the condition $P=(p,F^p)$ simply by the finite partial function $p$. Also for $A\subseteq L$, let $P\restrict A=(p\rest A,F^p)$ and $P\srestrict A=(p\restrict A, F^p\cap\What{A})$.
The definition is recursive and it is not clear to what extend it succeeds in defining a poset. However this will follow from Lemma~\ref{l.mainlemma}, stated below, which establishes not only transitivity but also a strong version of the complete embedding property, which is necessary for this definition to succeed. This Lemma is a generalization of the Main Lemma of~\cite{brendle03}. We note that if $A\in\mathcal I$ then it is clear from the definition that $\P_A\overset{\rm def}=\P(\mathcal T_A,\Q,\BbS)$ is a subset of $\P(\mathcal T,\Q,\BbS)$ and that the relation $\leq_{\P_A}$ is contained in $\leq_\P$. Clearly, the above definition also defines $\P_A=\P(\mathcal{T}_A,\Q,\BbS)$ for arbitrary $A\subseteq L$.

\begin{lemma}[Completeness of Embeddings]\label{l.mainlemma}
Let $\mathcal T=((L,\leq),\mathcal I, L_0, L_1)$ be a template, let $\Q$ be a finite function poset with $L_0=\dom(\Q)$ which satisfies the strong embedding property and let $\BbS$ be a good $\sigma$-Suslin poset. Let $B\in\mathcal{I}$, $A\subset B$ be closed. Then $\P_B$ is a partial order,  $\P_A\subset \P_B$ and even $\P_A\lessdot \P_B$. Furthermore, any $P=(p,F^p)\in \P_B$ has a canonical reduction $P_0=(p_0,F^{p_0})=p_0(P,A,B)\in\P_A$ such that
\begin{enumerate}[(i)]
{\item $\dom(p_0)=\dom(p)\cap A$, $F^{p_0}=F^p$,}
{\item $s^{p_0}_x=s_x^p$ for all $x\in \dom(p_0)\cap L_1$}
{\item $P\rest L_0=(p\rest L_0,F^{p_0})$ is a strong  $\Q_A$-reduction of $P\restrict L_0=(p\rest L_0, F^p)$}
\end{enumerate}
and such that whenever $D\in\mathcal{I}$, $B,C\subseteq D$, $C$ is closed, $C\cap B=A$ and $Q_0\in\P_C$ extends $P_0$, then there is $Q\in\P_D$ extending both $Q_0$ and $P$.
\end{lemma}

Lemma \ref{l.mainlemma} is proved by induction on the rank of $\mathcal T$. It uses the following lemmas, which are helpful for making simple manipulations with the conditions of $\P(\mathcal T,\Q,\BbS)$. In Lemmas~\ref{domain} through~\ref{l.paste} assume that $\mathcal{T}$, $\Q$ and $\BbS$ are as in Definition~\ref{template_poset} and that the Completeness of Embeddings Lemma~\ref{l.mainlemma} has been established for all templates of Rank $<\Rk(\mathcal{T})$. Let $\PP=\PP(\mathcal{T},\Q,\BbS)$.

\begin{lemma}\label{domain} If $P=(p,F^P)$ and $Q=(q,F^q)$ are conditions in $\P$ such that $\oc(P)$ and $\oc(Q)$ are contained in $L_x$ for some $x\in L_1$ and $P\leq_\P Q$, then there is $B\in \mathcal{I}_x$ such that $Q\leq_{\P_B} P$.
\end{lemma}
\begin{proof} If $x_p$ is defined and $x_p=x_q$ (resp. $x_p<x_q$) let $B^\prime\in\mathcal{I}_{x_p}$ (resp. $B^\prime
\in\mathcal{I}_{x_q}$) be a witness to $Q\leq_P P$. Using definition~\ref{d.template}.(2) find $B\in\mathcal{I}_x$ such that $B^\prime\subseteq B$ and
$\oc(P)\cup\oc(Q)\subseteq B$. Then $B^\prime\in\mathcal{I}_{B,x_p}$ (resp. $B^\prime\in\mathcal{I}_{B,x_q}$) is a witness to $Q\leq_{\P_B} P$. If $x_p$ is not defined and $B\in\mathcal{I}_x$ is such that $\oc(P)\cup\oc(Q)\subseteq B$, then since $Q\restrict L_0\leq_{\Q_B} P\restrict L_0$ we obtain $Q\leq_{\P_B} P$.
\end{proof}

\begin{lemma}\label{l.cut} Let $P=(p,F^p)$ and $Q=(q,F^q)$ be conditions in $\P$ and let $x_0\in L$. Then $Q\srestrict L_{x_0}\in\P$, $Q\srestrict L_{x_0}^=\in\P$ and if $Q\leq_\P P$ then $Q\srestrict L_{x_0}\leq_\P P\srestrict L_{x_0}$ and $Q\srestrict L_{x_0}^=\leq_{\P} P\srestrict L_{x_0}^{=}$.
\end{lemma}
\begin{proof}
The proofs of $Q\srestrict L_{x_0}\in\P$ and $Q\srestrict L_{x_0}\leq P\srestrict L_{x_0}$ proceed by induction on $n_q=|\dom(q)\cap L_1|$. The case $n_q=0$ follows by definition~\ref{finite_function_poset}. Thus suppose each of those is true whenever $n_q<n$ and let $n_q=n$.
To see that $Q\srestrict L_{x_0}\in\P$ note that if $x_q < x_0$ and $B$ is a witness to $Q\in\P$, then $B$ also witnesses $Q\srestrict L_{x_0}\in\P$. If $x_0\leq x_q$, then $n_{q\rest L_{x_0}} < n$ and so we can use the inductive hypothesis.

If $\dom(p\restrict L_{x_0})\subseteq L_0$, then $Q\srestrict L_{x_0}\leq_\P P\srestrict L_{x_0}$ follows from definition~\ref{finite_function_poset}. Suppose $n_{p\restrict L_{x_0}}\neq 0$  and let $B$ be a witness to $Q\leq P$. If $x_q<x_0$, then $B$ also witnesses $Q\srestrict L_{x_0}\leq P\srestrict L_{x_0}$. If $x_0<x_q$, then $Q\srestrict L_{x_q}\leq_{\P_B} P\restrict L_{x_q}$ and since $\leq_{\P_B}\subseteq\leq_{\P}$ we have that $Q\srestrict L_{x_q}\leq_\P P\srestrict L_{x_q}$. If $x_0=x_q$ we are done and if $x_0< x_q$ then $n_{q\rest L_{x_q}} <n$ and so by the inductive hypothesis, $Q\srestrict L_{x_0}\leq_\P P\srestrict L_{x_0}$.
\end{proof}

\begin{lemma}\label{ClaimX2} Let $P=(p,F^p)$ and $Q=(q,F^q)$ be conditions in $\P$. If $\dom(p)\subseteq\dom(q)$, $Q\restrict L_0\leq_\Q P\restrict L_0$ and $Q\srestrict {L_{x_p}^=}\leq_\P P\srestrict {L_{x_p}^=}$, then $Q\leq_\P P$.
\end{lemma}
\begin{proof}
Note that $x_q\geq x_p$. If $x_q=x_p$, then if $B$ is a witness to $Q\srestrict L_{x_p}^=\leq_\P P\srestrict L_{x_p}^=$, then $B$ is also a witness to $Q\leq P$. Thus suppose $x_q > x_p=x$. Let $(\dom(q)\cap L_1)\backslash L_{x_p}^==\{x_j\}_{j=1^,<}^n$ and let $H\in\mathcal{I}_{x_p}$ be a witness to $Q\srestrict L_{x_p}^=\leq_\P P\srestrict L_{x_p}^=$. In finitely many steps we can find an increasing sequence $\{H_j\}_{j=1}^n$ of elements of $\mathcal{I}$ such that forall $j$, $H_j\subseteq L_{x_j}$, $\oc(Q\srestrict L_{x_j})$, $\oc(P\srestrict L_{x_j})$ are contained in $H_j$ and $H_{j-1}$ is a witness to $Q\srestrict L_{x_j}\leq_{\P_{H_j}} P\srestrict L_{x_j}$ (taking $H_0=H$). Then $H_{n-1}$ is a witness to $Q\srestrict L_{x_n}\leq_{\P_{H_n}} P\srestrict L_{x_n}$, i.e. $H_n\in \mathcal{I}\cap\PO(L_{x_n})$ is a witness to $Q\leq_\P P$.
\end{proof}

\begin{lemma}\label{l.paste} Let $Q=(q,F^q)=Q\srestrict L_x$ be such that $x=\max\{\dom(q)\cap L_1\}$ be a condition in $\P$ with witness
$\bar{D}$. Let $P=(p,F^p)$ be a condition such that $(Q_0\srestrict L_x)\restrict L_0$ is a strong $\Q_{L_0\cap L_x}$-reduction of
$Q_0\restrict L_0$ and such that $Q\leq_\P Q_0\srestrict L_x^=$ with witness $\bar{D}$. Then $Q{\bar{\ltimes}}_x P=(q\bar{\ltimes}_x p, F^q\cup F^p)$
is a common extension of $Q$ and $P$ where $q\bar{\ltimes}_x p = q\cup p\restrict L\backslash L_x^=$.
\end{lemma}
\begin{proof} Since $q\bar{\ltimes}_x p\restrict L_0= q\restrict L_0\cup p\restrict L_0\backslash L_x^=$ and $(P\srestrict L_x)\restrict L_0$ is a strong
$\Q_{L_0\cap L_x}$-reduction of $P\restrict L_0$, we have that $(Q{\bar{\ltimes}}_x P)\restrict L_0\leq P\restrict L_0$. On the other hand $\dom(p\restrict L_0\backslash L_x^=)\cap \oc(Q)=\emptyset$ and so $(Q{\bar{\ltimes}}_x P)\restrict L_0\leq Q\restrict L_0$.

Suppose $n_p:=|\dom(p)\cap L_1\setminus L_{x_0}|$ is $0$. Then $\bar{D}$ witnesses that $Q{\bar{\ltimes}}_x P$ extends each of $P$ and $Q$. Now suppose
that the claim is true whenever $0\leq n_p <n$ and let $P$ be a condition with $n_p=n$. Then $x_p>x$ and $Q\leq (P\srestrict L_{x_p})\srestrict L_x^=$. By the inductive hypothesis $Q{\bar{\ltimes}}_x (P\srestrict L_{x_p})$ is a condition in $\P$ extending both $Q$ and $P\srestrict L_{x_p}$. By Lemma~\ref{domain} there is $B_0\in\mathcal{I}_{x_p}$ such that $Q\bar{\ltimes}_x(P\srestrict L_{x_p})\leq_{\P_{B_0}} Q,P\srestrict L_{x_p}$. Let
$B_1$ be a witness to $P\in\P$. Thus $P\srestrict L_{x_p}\in\P_{B_1}$ and $P\srestrict L_{x_p}\Vdash_{\P_{B_1}} p(x_p)\in\dot{\BbS}$. Then $B=B_0\cup B_1\in\mathcal{I}_{x_p}$ and $\P_{B_0}$, $\P_{B_1}$ completely embed into $\P_B$. This implies that $Q\bar{\ltimes}_x{(P\srestrict L_{x_p})}\leq_{\P_B} Q, P\srestrict L_{x_p}$ and so in particular $Q\bar{\ltimes}_x (P\srestrict L_{x_p})\Vdash_{\P_B} p(x_p)\in\dot{\BbS}$. Then $B$ is also a witness to $Q\bar{\ltimes}_x P \leq_\P P$. Since $x_q< x_{q\bar{\ltimes}_x p}=x_p$, the set $B_0$ is a witness to $Q\bar{\ltimes}_x P \leq_\P Q$
\end{proof}

\begin{proof}[Proof of Lemma \ref{l.mainlemma}]
We establish the Lemma by recursion on the rank of the underlying template.  The $\Rk(\mathcal T_B)=0$ case is clear. So assume that the Lemma holds for all templates of rank $<\alpha$, and let $\Rk(\mathcal{T}_B)=\alpha$. Let $\P=\P_B$.

\bigskip
\noindent
{\bf{Transitivity:}} To see that $\leq_{\P}$ is transitive, fix $P_0,P_1,P_2\in\P$ such that $P_1\leq_{\P} P_0$ and $P_2\leq_{\P} P_1$, and assume that $x_{p_0}$ is defined (since otherwise there is nothing to show.) Fix witnesses $B_1\in \mathcal I_{x_{p_1}}$ and $B_2\in\mathcal I_{x_{p_2}}$ to $P_1\leq_{\P} P_0$ and $P_2\leq_{\P} P_1$. Since $\Dp(B_1\cup B_2)<\alpha$, the inductive hypothesis gives that $\P_{B_1},\P_{B_2}\lessdot \P_{B_1\cup B_2}$, and so we have $P_i\restrict L_{x_{p_2}}=P_i\restrict B_1\cup B_2\in\P_{B_1\cup B_2}$ for $0\leq i\leq 2$, and that
$$
P_2\srestrict L_{x_{p_2}}\leq_{\P_{B_1\cup B_2}} P_1\srestrict L_{x_{p_2}}\leq_{\P_{B_1\cup B_2}} P_0\srestrict L_{x_{p_2}}.
$$
Thus by the inductive hypothesis we have $P_2\srestrict L_{x_{p_2}}\leq_{\P_{B_1\cup B_2}} P_0\srestrict L_{x_{p_2}}$. If $x_{p_0}<x_{p_2}$ then it now follows from the definition of $\leq_{\P}$ that $P_2\leq_{\P} P_0$. So assume that $x_{p_0}=x_{p_2}$. It is clear that $p_i(x_{p_2})$ is a $\P_{B_1\cup B_2}$-name for $0\leq i\leq 2$. Since $\P_{B_1},\P_{B_2}\lessdot \P_{B_1\cup B_2}$ we must have that $P_1\srestrict L_{x_{p_2}}\forces_{\P_{B_1\cup B_2}} p_1(x_{p_2})\leq_{\dot\BbS} p_0(x_{p_2})$ and $P_2\srestrict L_{x_{p_2}}\forces_{\P_{B_1\cup B_2}} p_2(x_{p_2})\leq_{\dot\BbS} p_1(x_{p_2})$. But then $P_2\restrict L_{x_{p_2}}\forces_{\P_{B_1\cup B_2}} p_1(x_{p_2})\leq_{\dot\BbS} p_0(x_{p_2})$ and so $P_2\srestrict L_{x_{p_2}}\forces_{\P_{B_1\cup B_2}} p_2(x_{p_2})\leq_{\dot\BbS} p_0(x_{p_2})$. Thus
$$(P_2\srestrict L_{x_{p_2}}, p_2(x_{p_2}))\leq_{\P_{B_1\cup B_2}*\dot\BbS} (P_0\srestrict L_{x_{p_2}}, p_0(x_{p_2}))$$
as required.

\bigskip
\noindent
{\bf{Suborders:}} Let $A\subset B$ be closed, $B\in\mathcal{I}$ be given. We will show that $\P_A\subset \P_B$. Assume $R=(r,F^r)\in\P_A$. Let $x=x_r$. By definition of the iteration there is $\bar{A}\in{(\mathcal{I}\restrict A)}_x$ such that $R\srestrict{(A\cap L_x)}\in\P_{\bar{A}}$ and $\dot{f}_x^r$ is a $\P_{\bar{A}}$-name.

Note that $\bar{A}\in \mathcal{I}\restrict A$ means that there is $B_0\in\mathcal{I}$ such that $\bar{A}=B_0\cap A$. On the other hand $A\subset B$, so $\bar{A}\subset B$ and so $B_0\cap A=B_0\cap B\cap A$. But $\mathcal{I}$ is closed under finite intersections and so $B_0\cap B\in\mathcal{I}$, even $B_0\cap B\in\mathcal{I}_B$. So without loss of generality there is $\bar{B}\in\mathcal{I}_ B$ (just take $\bar{B}=B_0\cap B$) such that $\bar{A}=A\cap \bar{B}$. Since $\bar{A}\subseteq L_x$, $x\notin \bar{B}$. Then by definition~\ref{d.template}.(3), $\bar{B}\cap L_x\in\mathcal{I}_B$. Therefore we can assume that $\bar{B}\subseteq L_x$. Thus $\bar{B}\subset B$ and $\hbox{Dp}(\bar{B})<\hbox{Dp}(B)=\alpha$. By the inductive hypothesis, $\P_{\bar{A}}\subseteq\P_{\bar{B}}$ and $\P_{\bar{A}}\lessdot \P_{\bar{B}}$. Therefore $\dot{f}^r_x$ is a $\P_{\bar{B}}$-name as well. Thus $R\srestrict L_x\in\P_{\bar{B}}$ and $\dot{f}^r_x$ is a $\P_{\bar{B}}$-name. That is, $R\in\P_B$.

\bigskip
\noindent
{\bf{Complete Embeddings:}} Assume $P=(p,F^p)\in \P_B$. We will construct a "canonical reduction" $P_0=p_0(P,A,B)$. Let $x=x_p$. By definition of the iteration, there is $\bar{B}\in\mathcal{I}_{B,x}$ such that $P\srestrict L_x=\bar{P}\in\P_{\bar{B}}$ and $\dot{f}^p_x$ is a $\P_{\bar{B}}$-name. Let $\bar{A}=A\cap \bar{B}$. Then $\bar{A}\in\mathcal{I}\restrict A$, $\bar{A}\subset \bar{B}$, $\bar{A}\in\PO(L_x)$.
Repeating the argument from $(2)$, we obtain  $\P_{\bar{A}}\lessdot \P_{\bar{B}}$. Therefore $\bar{P}$ has a "canonical reduction" $\bar{P}_0=p_0(\bar{P},\bar{A},\bar{B})$. Let $F^{p_0}= F^p\cap\What{A}$. Define $p_0\restrict L_0$ so that $(p_0\restrict L_0, F^{p_0})$
is a strong $\Q_A$-reduction of $(p_0\restrict L_0, F^p)$ and $p_0\restrict L_0\cap\bar{A}\supseteq \bar{p}_0\restrict L_0$. Let $p_0\restrict L_1\cap L_x=
\bar{p}_0\restrict L_1\cap L_x$. Then $P_0\srestrict L_x\leq_{\P_{\bar{A}}}\bar{P}_0$ and so $P_0\srestrict L_x$ is a canonical reduction of $\bar{P}$.
We can assume that $p(x)$ is a nice name for a condition in $\BbS$ below $\bar{P}$. If $x\notin A$, let $p_0(x)=p(x)$ and if $x\in A$ let $p_0(x)$ be a canonical projection of $p(x)$ below $P_0\srestrict L_x$.

\bigskip
\noindent
Now assume $D\in\mathcal{I}$, $C\subseteq D$ closed are such that $B\cup C\subseteq D$, $A=B\cap C$ and $\Dp(D)=\alpha$.
Let $Q_0=(q_0, F^{q_0})\leq_{\P_C} P_0$. We will construct a common extension of $Q_0$ and $P$.

{\bf{Case 1:}} $x\notin A$. Then clearly $x\notin C$. Let $y=\max(\dom(q_0)\cap L_x\cap L_1)$. Then $y<x$. By Lemma~\ref{l.cut} $Q_0\restrict L_y^=\leq_{\P_C} P_0\restrict L_y^=$ and so there is $\bar{E}\in(\mathcal{I}\restrict C)_y$ witnessing this fact. Using~\ref{d.template}.(2) find $\bar{F}\in\mathcal{I}_{D,y}$ such that
$\bar{E}=\bar{F}\cap C$. By~\ref{d.template}.(3) there is $\bar{G}\in\mathcal{I}_{D,x}$ such that
$\dom(q_0)\cap L_x\backslash L_y\subseteq\bar{G}$. Let $\bar{D}=\bar{B}\cup \bar{F}\cup\bar{G}$, $\bar{C}=(\bar{G}\cap C)\cup \bar{E}\cup \bar{A}$ and note that
$\bar{D}\in\mathcal{I}_{D,x}$, $\bar{C}\in(\mathcal{I}\restrict C)_x$. Clearly $\bar{C}\subseteq\bar{D}$, $\bar{C}\cap\bar{B}=\bar{A}$.

Note that $\bar{Q}_0:=Q_0\srestrict L_x\leq_{\P_{\bar{C}}} P_0\srestrict L_x$ with witness $\bar{E}$ (observe that $\bar{E}$ also belongs to $(\mathcal{I}\restrict \bar{C})_y$).  Passing to an extension if necessary, we can assume that $\bar{Q}_0\restrict L_0$ is a strong $\Q_{\bar{C}}$-reduction of $Q_0\restrict L_0$. Since $\hbox{Dp}_{\mathcal{I}\rest C}(\bar{C})\leq\hbox{Dp}_{\mathcal{I}}(\bar{D})<\hbox{Dp}_\mathcal{I} (D)=\alpha$, we can apply the inductive hypothesis to $\bar{A}$, $\bar{B}$, $\bar{C}$, $\bar{D}$. Thus there is a common extension
$\bar{Q}=(\bar{q},F^{\bar{q}})\leq_{\P_{\bar{D}}}\bar{Q}_0, P\srestrict L_x$. With this we are ready to define a common extension $Q=(q,F^q)$ of $Q_0$ and $P$ as follows:

Let $q^\prime=\bar{q}\cup\{(x,p(x))\}$, $F^{q^\prime}=F^{\bar{q}}$ and let $Q^\prime=(q^\prime, F^{q^\prime})$.
Then $\bar{D}$ does not only witness $Q^\prime\in\P_D$, but also $Q^\prime\leq_{\P_D} \bar{Q}_0=Q_0$.
By Lemma~\ref{l.paste} $Q^{\prime\prime}:=Q^\prime{\bar{\ltimes}}_x Q_0$ is a common extension in $\P_D$ of $Q^\prime$ and $Q_0$. Denote $Q^{\prime\prime}=(q^{\prime\prime}, F^{q^\prime}\cup F^{q_0})$ and let  $\hat{p}=p\restrict L_0\backslash \dom(q^{\prime\prime})$.
Let $q=q^{\prime\prime}\cup\hat{p}$, $F^q=F^{q^\prime}\cup F^{q_0}$ and let $Q=(q,F^q)$.
Since $\oc(Q^{\prime\prime})\cap\dom(\hat{p})=\emptyset$, we obtain that $Q=(q^{\prime\prime}\cup\hat{p}, F^{q_0}\cup F^p)$ is a condition in $\P$, extending $Q^{\prime\prime}$. Thus in particular $Q\leq Q_0$.

To see that $Q\leq P$, first observe that $Q^{\prime\prime}\restrict L_0 \leq Q_0\restrict L_0\leq P_0\restrict L_0$ and since by definition $P_0\restrict L_0$ is a strong $\Q_A$-reduction of $P\restrict L_0$, we obtain $(q^{\prime\prime}\restrict L_0\cup\hat{p}\restrict L_0, F^p)\leq P\restrict L_0$. But then $Q\restrict L_0\leq P\restrict  L_0$, $Q\srestrict L_{x_p}^=\leq P\srestrict L_{x_p}^=$ and $\dom(p)\subseteq \dom(q)$, which by Lemma~\ref{ClaimX2} gives $Q\leq P$.

{\bf{Case 2: $x\in A$.}}
Then $x\in C$. Let $\bar{C}\in(\mathcal{I}\restrict C)_x$ be a witness to $Q_0\srestrict L_x^=\leq_{\P_C} P_0\srestrict L_x^=$. That is $\bar{Q}_0=Q_0\srestrict L_x\leq_{\P_{\bar{C}}} P_0\srestrict L_x(\leq\bar{P}_0)$ and $Q_0\srestrict L_x\Vdash " (\check{s}_x^{q_0},\dot{f}^{q_0}_x)\leq_{\dot{\BbS}}(s^{p_0}_x,\dot{f}^{p_0}_x)"$.
By definition $\bar{A}=A\cap\bar{B}$, where $\bar{B}\in\mathcal{I}_{B,x}$. Also by definition of $\bar{C}\in(\mathcal{I}\restrict C)$ there is $C_0^\prime\in\mathcal{I}$ such that $\bar{C}=C_0^\prime\cap C$. Then $x\notin C_0^\prime$ and so by~\ref{d.template}.(3) $C_0=C_0^\prime\cap L_x\in\mathcal{I}_x$ and $\bar{C}= C_0\cap C$. Passing to an extension if necessary, we can assume that $\bar{Q}_0\restrict L_0$ is a strong $\Q_{\bar{C}}$-reduction of $Q_0\restrict L_0$.

Then $\bar{A}\cup\bar{C}\in (\mathcal{I}\restrict C)_x$ and since
$\hbox{Rk}(\mathcal{T}_{\bar{A}\cup\bar{C}})<\hbox{Rk}(\mathcal T)$, we have $\P_{\bar{C}}\lessdot \P_{\bar{A}\cup\bar{C}}$. Therefore $\dot{f}^{q_0}_x$ is also a $\P_{\bar{A}\cup\bar{C}}$-name and  so without loss of generality, we may assume that $\bar{A}\subseteq\bar{C}$.
Observe that $\bar{A}=\bar{C}\cap \bar{B}$. Note also that $\bar{D}:= D\cap C_0\in\mathcal{I}_{D,x}$ and $\bar{C}=\bar{D}\cap C$. We may also assume that
$\bar{B}\subseteq\bar{D}$ (otherwise take $\bar{B}\cup\bar{D}\in\mathcal{I}_{D,x}$). Since $\hbox{Dp}_{\mathcal{I}}(\bar{D}) < \hbox{Dp}_{\mathcal I}(D)=\alpha$, we can use the inductive hypothesis when working with $\bar{A}$, $\bar{B}$, $\bar{C}$, $\bar{D}$.

Let $n$ be such that $\BbS$ is $n$-Suslin. Let $m=|s^{q_0}_x|$. Find $\hat{Q}_0\leq_{\P_{\bar{C}}} \bar{Q}_0$ and $s^\prime\in{^{n\cdot m}\omega}$  such that $\hat{Q}_0\Vdash"\dot{f}^{p_0}_x\rest n\cdot m = \check{s}^\prime".$ Let $G$ be $\P_{\bar{C}}$-generic filter such that $\hat{Q}_0\in G$. Now note that $\dot{f}^{p_0}_x$ is a $\P_{\bar{A}}$-name
and $\P_{\bar{A}}\lessdot\P_{\bar{C}}$ by the inductive hypothesis (here we use the fact that $\hbox{Dp}_{\mathcal{I}\rest{\bar{C}}}(\bar{C})\leq \hbox{Dp}_{\mathcal{I}}(\bar{D})<\alpha$). Therefore $G\cap\bar{A}$ is a $\P_{\bar{A}}$-generic and there is $U\in G\cap \bar{A}$ such that $U\Vdash_{\P_{\bar{A}}}\dot{f}^{p_0}_x\restrict n\cdot m=\check{s}^\prime$. Now $U, P_0\srestrict L_x\in G\cap \bar{A}$, so they have a common extension $E^\prime\in G\cap \bar{A}$ and $E^\prime\Vdash_{\P_{\bar{A}}} \dot{f}^{p_0}_x\restrict n\cdot m=\check{s}^\prime$. Since $E^\prime$, $\hat{Q}_0$ are in $G$ they have a common extension $\hat{\hat{Q}}_0\in G$ (and so in $\P_{\bar{B}}$). Then in particular $\hat{\hat{Q}}_0\leq E^\prime$ and so $\hat{\hat{Q}}_0$ has a reduction $\tilde{Q}_0$ in $\P_{\bar A}$ which extends $E^\prime$. Thus $\tilde{Q}_0\Vdash_{\P_{\bar{A}}} \dot{f}^{p_0}_x\restrict n\cdot m=\check{s}^\prime$ and $\tilde{Q}_0\leq P_0\srestrict L_x$. But then $\tilde{Q}_0$ is compatible in $\P_{\bar{A}}$ with some element  $a\in\mathcal{A}_{n\cdot m}$. Here following the notation of Lemma~\ref{l.projection_name_real}, we assume that $\dot{f}^p_x=\{(b,s(b))\}_{b\in\mathcal{B}_,n\geq 1}$ and $\dot{f}^{p_0}_x=\{(a,s(a))\}_{a\in\mathcal{A}_n, n\geq 1}$. Since $a\Vdash_{\P_{\bar{A}}} \dot{f}^{p_0}_x\restrict n\cdot m=\check{s(a)}$ and $a$ is compatible with $\tilde{Q}_0$, it must be the case that $s(a)=s^\prime$. Let
$P_0^*$ be a common $\P_{\bar{A}}$ extension of $a$ and $\tilde{Q}_0$. Then $P_0^*\leq a$ and $P_0^*$ is a reduction of $\hat{\hat{Q}}_0$ (since $\tilde{Q}_0$ is such a reduction; also $P_0^*$ is a reduction of $\bar{Q}_0$).  By construction $a$ is a reduction of some condition $b\in\mathcal{B}_{n\cdot m}$ such that $s(b)=s(a)$, i.e. $b\leq \bar{P}$ and $b\Vdash_{\P_{\bar{B}}} \dot{f}^p_x\restrict n\cdot m=\check{s}^\prime$. Then $P^*_0$ is compatible with $b$, with common extension $\bar{P}^+$. By the inductive hypothesis $\P_{\bar{A}}\lessdot\P_{\bar{B}}$ and so $\bar{P}^+$ has a canonical reduction $\hat{P}^+=p_0(\bar{P}^+,\bar{A},\bar{B})$. By Lemma~\ref{l.strongsuslin}, $\hat{P}^+$ is compatible with $P_0^*$ (since $\bar{P}^+\leq P_0^*$ and every canonical reduction is clearly also a reduction). Therefore they have a common extension $\bar{P}^+_0$. Note that $\bar{P}_0^+\leq P_0^*$ and $\bar{P}_0^+$ is a canonical reduction of $\bar{P}^+$. Since $P_0^*$ is a reduction of $\hat{\hat{Q}}_0$ onto $\P_{\bar{A}}$, there is $\bar{Q}_0^+\in\P_{\bar{C}}$ extending $\bar{P}_0^+$ and $\hat{\hat{Q}}_0$. Now using the fact that $\bar{Q}_0^+\leq\bar{P}_0^+$ and $\bar{P}_0^+$ being a canonical reduction of $\bar{P}^+$, we obtain a condition $T=(t,F^t)\in\P_{\bar{D}}$ such that $T\leq_{\P_{\bar{D}}} \bar{P}^+$ and $T\leq_{\P_{{\bar{D}}}} \bar{Q}_0^+$.

Then
$$T\Vdash_{\P_{\bar{D}}} "(s^{q_0}_x,\dot{f}^{q_0}_x)\leq_{\dot{\BbS}} (s^{p_0}_x,\dot{f}^{p_0}_x)\wedge (s^p_x,\dot{f}^p_x)\,\hbox{is such that}\; s^p_x=s^{p_0}_x\wedge\dot{f}^p_x\restrict n\cdot m=\dot{f}^{p_0}_x\restrict n\cdot m".$$
Since $\BbS$ is by assumption $n$-Suslin we have
$T\Vdash_{\P_{\bar{D}}} \exists t(x)\in \dot{\BbS}(t(x)\leq_{\dot{\BbS}} q_0(x), p(x))$.
Find $\bar{Q}^+\leq T$ and a nice name $(\check{s}^q_x,\dot{f}^q_x)$ for a condition in $\BbS$ below $\bar{Q}^+$ such that
$\bar{Q}^+\Vdash_{\P_{\bar{D}}} "(\check{s}^q_x,\dot{f}^q_x)\leq_{\dot{\BbS}} (s^{q_0}_x,\dot{f}^{q_0}_x), (s^p_x,\dot{f}^p_x)"$.
Denote $\bar{Q}^+=(\bar{q}^+, F^{{\bar{q}}^+})$.

With this we are ready to define a common extension $Q=(q,F^q)$ of $Q_0$ and $P$. Let $q^\prime=\bar{q}^+\cup\{(x,q(x))\}$, $F^{q^\prime}=F^{\bar{q}}$ and  $Q^\prime=(q^\prime, F^{q^\prime})$.  Given $Q^\prime$, define $Q^{\prime\prime}$, $\hat{p}$ and $Q$ as in Case 1. Then following the proof of Case 1,
one obtains that $Q$ is a common extension of $Q_0$ and $P$.
\end{proof}

\subsection{Basic properties of the iteration}

Having established our generalized ``Main Lemma'', we now proceed to develop the remaining basic tools that we need to work with the iteration along a two-sided template. These steps are parallel to those taken in Brendle \cite[pp. 2640--2642]{brendle03}, and we provide complete proofs only where it seems warranted.
For the discussion in this section fix $\mathcal T$, $\Q$ and $\BbS$ as in Lemma \ref{l.mainlemma}.

\begin{lemma}\label{l.template_knaster}
Suppose $\Q$ is Knaster. Then $\P(\mathcal T,\Q,\BbS)$ is Knaster.
\end{lemma}
\begin{proof} Let $\langle q_\alpha:\alpha<\omega_1\rangle$ be an arbitrary sequence of conditions in $\P$. Since $\Q$ is Knaster
we can assume that $\la Q_\alpha\restrict L_0:\alpha<\omega_1\ra$ are pairwise compatible in $\Q$. Applying the $\Delta$-system lemma and the fact that $\Q$ is Knaster, we can assume that for all distinct $\alpha,\beta <\omega_1$
$\dom(q_\alpha)\cap\dom(q_\beta)=F$ for some fixed finite set $F\subseteq L$ and that $Q_\alpha\restrict L_0$, $Q_\beta\restrict L_0$ are compatible. Furthermore we can assume that for all $x\in F\cap L_1$ there are $s_x\in{^{<\omega}\omega}$, $t_x\in{^{n\cdot |s_x|}\omega}$ such that if $B$ is a witness to $Q_\alpha\restrict L_x^=\in\P$, then
$Q_\alpha\restrict L_x\Vdash_{\P_B}\pi_0(q_\alpha(x))=\check{s}_x\wedge \pi_1(q_\alpha(x))\restrict n\cdot |s_x|=\check{t}_x$.

Fix $\alpha,\beta$ distinct. We will show that $Q_\alpha$, $Q_\beta$ are compatible in $\P$. Let $\{x_i\}_{i\in m}$ enumerate in $<_L$-increasing order
$(\dom(q_\alpha)\cup \dom(q_\beta))\cap L_1$, and let $R=(r,F)$ be a common extension of $Q_\alpha\restrict L_0$ and $Q_\beta\restrict L_0$. Passing to an extension if necessary, we can assume that $R\srestrict L_{x_0}$ is a strong $\Q_{L_0\cap L_{x_0}}$-reduction of $R$. Furthermore there are $R^*_0\leq_{\Q_{L_{x_0}}} R\srestrict L_{x_0}$ and $t(x_0)$ such that  $R^*_0\Vdash_{\Q_{L_{x_0}}} t(x_0)\leq q_\alpha(x_0),\; q_\beta(x_0)$.
Let $R^*=(r^*,F^*)$ and let $R_0=(r_0, F_0)=(r^*_0\cup\{(x_0, t(x_0))\}\cup r\restrict L\backslash L_{x_0}, F^*_0\cup F)$. Since $R\srestrict L_{x_0}$ is a strong $\Q_{L_{x_0}}$-reduction of $R$, we obtain that $R_0\leq_{\P} R$. Furthermore $R_0\srestrict L_{x_1}$ is a common extension of $Q_\alpha\srestrict L_{x_1}$ and $Q_\beta\srestrict L_{x_1}$ (in $\P$).

Suppose for some $i< m-1$ we have a condition $R_i=(r_i, F_i)\leq_\P R$ such that $r_i\restrict L\backslash L_{x_i}^==r\restrict L\backslash L_{x_i}^=$,
$(R_i\srestrict L_{x_i})\srestrict L_0$ is a strong $\Q_{L_{x_i}}$-extension of $R_i$ and
$R_i\srestrict L_{x_i}\leq_{\P_{L_{x_i}}} Q_\alpha\srestrict L_{x_i}, Q_\beta\srestrict L_{x_i}$.
Then we can find an extension $R_i^*$ of $R_i\srestrict L_{x_i}$ in $\P_{L_{x_i}}$ and name $t(x_{i+1})$ such that
$R_i^*\Vdash_{\P_{L_{x_i}}} t(x_{i+1})\leq q_\alpha(x_{i+1}),q_\beta(x_{i+1})$.
Let $R_i^*=(r_i^*, F_i^*)$. Since $(R_i\srestrict L_{x_i})\restrict L_0$ is a strong $\Q_{L_{x_i}}$-reduction of $R_i$,
we obtain $R_{i+1}=(r_{i+1},F_{i+1})= (r_i^*\cup\{(x_{i+1}, t(x_{i+1}))\}\cup r\restrict L\backslash L_{x_{i+1}}, F_i^*\cup F_i)\leq_\P R_i$
and $R_{i+1}\srestrict L_{x_{i+2}}\leq Q_\alpha\srestrict L_{x_{i+2}}, Q_\beta\srestrict L_{x_{i+2}}$. Then for  $i=m$, we obtain $R_m\leq_{\P} Q_{\alpha},Q_{\beta}$.
\end{proof}

We omit the proofs of the next three Lemmas since they follow very closely the proofs of \cite[Lemma 1.3, Lemma 1.4 and Lemma 1.5]{brendle03}

\begin{lemma}\label{l.Sembed}
Let $x\in L_1$, $A\in\mathcal I_x$. Then the two-step iteration $\P_A*\BbS$ completely embeds into $\P$.
\end{lemma}

\begin{lemma}\label{l.names}
For any $p\in\P(\mathcal T,\Q,\BbS)$ there is a countable set $A\subseteq L$ such that $p\in\P_{\cl(A)}$. Similarly, if $\tau$ is a $\P$-name for a real then there is a countable $A\subseteq L$ such that $\tau$ is a $\P_{\cl(A)}$-name.
\end{lemma}

\begin{lemma}
Let $\mathcal J\subseteq \mathcal I$ be such that $\mathcal T_{\mathcal J}=((L,\leq),\mathcal J,L_0,L_1)$ is a template. Suppose $\mathcal J$ is cofinal in $\mathcal I$. Then $\P(\mathcal T_{\mathcal J},\Q,\BbS)$ is forcing equivalent to $\P(\mathcal T,\Q,\BbS)$.
\end{lemma}

\section{$\mathfrak a_g$ can be $\aleph_\omega$}

We now start working towards the main theorem of the paper. The model in which $\hbox{cof}(\mathfrak a_g)=\omega$ is obtained by forcing with a poset of the form $\P(\mathcal T,\Q,\BbS)$, where $\Q$ being the poset $\Q_{L_0}$ that adds a cofinitary group with $L_0$ generators, $\BbS$ be localization forcing, and $\mathcal T$ is the particular template used by Brendle in \cite{brendle03}.

\subsection{Basic estimates for $\mathfrak a_g$} Before specifying $\mathcal T$, we prove two generally applicable Lemmas, which are parallel to \cite[Proposition 1.6 and Proposition 1.7]{brendle03}.

\begin{lemma}\label{l.nonM}
Let $\mathcal T$ be a template, let $\Q$ be a finite function poset with the complete embedding property and $L_0=\dom(\Q)$, let $\BbS=\BbL$ be localization forcing, and let $\mu$ be a regular uncountable cardinal. Suppose $\mu\subseteq L_1$ (as an order), that $\mu$ is cofinal in $L$, and that $L_\alpha\in\mathcal I$ for all $\alpha<\mu$. Then $\P(\mathcal T,\Q,\BbS)$ forces that $\non(\mathcal M)=\mu$ and $\mathfrak a_g\geq\mu$.
\end{lemma}
\begin{proof} 
Let $G$ be $\P(\mathcal T,\Q,\BbL)$-generic over $V$ and work in $V[G]$. Let $\phi_\alpha$ be the slalom added in coordinate $\alpha<\mu$ (this makes sense by Lemma \ref{l.Sembed}.) Since $\mu$ is regular and uncountable and is cofinal in $L$ it is clear by Lemma \ref{l.names} that the family $\langle\phi_\alpha:\alpha<\mu\rangle$ localizes all reals $V[G]$ (since any real must appear in some $V[G\cap\P_{L_{\alpha}}]$ for some $\alpha<\mu$.) Thus $\cof(\mathcal N)\leq\mu$. On the other hand, if $F\subseteq\omega^\omega$ is a family of size $<\mu$ in $V[G]$, then there must be some $\alpha<\mu$ such that all reals of $F$ already are in $V[G\cap\P_{L_\alpha}]$, and so $\phi_\alpha$ localizes all reals in $F$. Thus $\add(\mathcal N)\geq\mu$. Therefore $\non(\mathcal{M})=\mu$ and so by Theorem \ref{t.agM} we have $\mathfrak a_g\geq\mu$.
\end{proof}

\begin{lemma}\label{l.mcgL0}
Let $\mathcal T$ be a template, and let $\Q=\Q_{L_0}$ be the poset for adding a cofinitary group with $L_0$ generators. Suppose that $L$ has uncountable cofinality and that $L_0$ is cofinal in $L$. Then $\P(\mathcal T,\Q,\BbS)$ adds a maximal cofinitary group of size $|L_0|$.
\end{lemma}
\begin{proof}
Let $G$ be $\P=\P(\mathcal T,\Q,\BbS)$-generic. Let $\rho_G:L_0\to S_\infty$ be defined as follows: for every $x\in L_0$ let
$\rho_G(x)=\{s^p_x: P\in G\wedge P\rest L_0=(s^p, F^p)\}$. Note that $\rho_G=\bigcup\{s^p_x: P\in G\cap \P_{L_0}\}$ and so
by Proposition \ref{p.rhoG} the function $\rho_{G}$ induces a cofinitary representation $\hat\rho_{G}$ of $\F_{L_0}$. We will show that
$\im(\rho_{G})$ is a maximal cofinitary group (which then clearly has size $|L_0|$.)

Suppose not. Then there is a permutation $\sigma\in\cofin(S_\infty)$ and $b_0\notin L_0$ such that $\rho_{G}':L_0\cup\{b_0\}\to S_\infty$, defined by $\rho'_{G}\restrict L_0=\rho_{G}$ and $\rho'_{G}(b_0)=\sigma$, induces a cofinitary representation. Let $\dot\sigma$ be a $\P$-name for $\sigma$ in $V$. Then by Lemma \ref{l.names} there is a countable set $A\subseteq L$ such that $\dot\sigma$ is a $\P_{\cl(A)}$-name. Since $L_0$ is cofinal in $L$ and $L$ has uncountable cofinality, there is some $x\in L_0$ such that $\cl(A)\subseteq L_x$ and so $\P_{\cl(A)}\lessdot \P_{L_x}$. Let $G_0=G\cap \P_{L_0}$ and $H=G\cap \P_{P_x}$.

\begin{clm} $V[H]\vDash "D_{\sigma,N}=\{P\in(\PP/H): \exists n\geq N(s^p_x(n)=\sigma(n))\;\hbox{where}\; P\rest L_0=(s^p, F^p)\}\;\hbox{is dense}."$.
\end{clm}
\begin{proof} Let $P_0\in (\P/H)$. Thus $P\restrict L_0\cap L_x\in H_0:=G\cap \P_{L_0\cap L_x}$.  By Lemma \ref{cofinlemma} the set
$$V[H_0]\vDash D_{\sigma, N,x}^0=\{p\in (\Q_{L_0}/\Q_{L_x\cap L_0}) : (\exists n\geq N) s^p_x(n)=\sigma(n)\;\hbox{is dense}\}.$$
Thus there is $(t,E)\leq (s^{p_0}\rest L_0\backslash L_x, F^{p_0})$ such that $(t,E)\in D^0_{\sigma,N}$ i.e. $t_x(n)=\sigma(n)$ for some $n\geq N$.
Define $P_1\in \P/H$ as follows: $P_1\rest L_x= P_0\rest L_x$, $P_1\rest (L_0\backslash L_x)=(t,E)$, $P_1\rest L_1\backslash L_x=P_0\rest L_1\backslash L_x$. Then in $V[H]$ we have $P_1\leq P_0$ and $P_1\in D_{\sigma,N}$.
\end{proof}
Then in $V[G]$ there are infinitely many $n$ such that $\sigma(n)=\sigma_x(n)$, contradicting the fact that $\rho^\prime_G$ induces a cofinitary representation.
\end{proof}

\begin{section}{The Isomorphism of Names Argument}

Until the end of the paper assume CH. We will use the template construction developed by J. Brenlde and S. Shelah to show that the minimal
size of a maximal almost disjoint family can be of countable cofinality (see~\cite{brendle03}).

Let $\lambda$ be a cardinal of countable cofinality and more precisely, let $\lambda=\bigcup_{n\in\omega}\lambda_n$,
where $\{\lambda_n\}_{n\in\omega}$ is a strictly increasing sequence of regular cardinals, $\lambda_0\geq\aleph_2$,
$\lambda_n^{\aleph_0}=\lambda_n$ for all $n$, and $\kappa^{\aleph_0}<\lambda_n$ for $\kappa<\lambda_n$.  In the following,
let $\mu^*$ denote a disjoint copy of $\mu$, with the reverse ordering.  Let $<_{\mu}$ denote the ordering of $\mu$. We will
refer to the elements of $\mu$ as {\emph{positive}} and to the elements of $\mu^*$ as {\emph{negative}}.
If $\alpha\neq \beta\in \lambda^*\cup\lambda$, we will say that $\alpha<_{\lambda^*\cup\lambda} \beta$,
if either $\alpha\in\lambda^*$ and $\beta\in\lambda$, or both are in $\lambda$ and $\alpha<\beta$, or both are
in $\lambda^*$ and $\alpha<_{\lambda^*}\beta$. For each $n$ fix a partition
$\lambda_n^*=\bigcup_{\alpha<\omega_1} S^\alpha_n$, where the $S^\alpha_n$'s are co-initial in $\lambda_n^*$
and for $m<n$, $S^\alpha_n\cap\lambda^*_m=S^\alpha_m$. Definitions 5.1 - 5.5 and Lemma 5.6 can be found in~\cite{brendle03}.

\begin{definition} Let $L=L(\lambda)$ consist of all finite, nonempty sequences $x$ such that
\begin{enumerate}
{\item $x(0)\in \lambda_0$,}
{\item $x(n)\in \lambda_n^*\cup\lambda_n$ for $0<n<|x|-1$,}
{\item for $|x|\geq 2$, if $x(|x|-2)$ is positive, then
$x(|x|-1)\in\lambda_{|x|-1}^*\cup \lambda$ and if $x(|x|-2)$ is negative, then $x(|x|-1)\in \lambda^*\cup \lambda_{|x|-1}$.}
\end{enumerate}
Whenever $x,y\in L$ let $x<y$ if and only if
\begin{enumerate}
{\item either $x\subset y$ and $y(|x|)$ is positive,}
{\item or $y\subset x$ and $x(|y|)$ is negative,}
{\item or $n=\hbox{min}\{k:x(k)\neq y(k)\}$ is defined and
$x(n)<_{\lambda^*\cup\lambda} y(n)$.}
\end{enumerate}
\end{definition}

Clearly $(L,<)$ is a linear order.

\begin{definition} Let $L_1=\{x\in L: |x|=1\;\hbox{or}\; x(|x|-1)\in\lambda^*_{|x|-1}\cup\lambda_{|x|-1}\}$
and let $L_0=L\backslash L_1$.
\end{definition}

\begin{remark}
Note that $x\in L_0$ if and only if $|x|\geq 2$ and if $x(|x|-2)$ is positive, then $x(|x|-1)\in [\lambda_{|x|-1},\lambda)$
and if $x(|x|-2)$ is negative, then $x(|x|-1)\in (\lambda^*,\lambda^*_{|x|-1}]$. Note also that both $L_0$ and $L_1$ are
cofinal in $L$, and that neither of them is of countable cofinality.
\end{remark}

\begin{definition} Let $L_{\hbox{rel}}$ be the subset of $L_1$ of all $x$ such that $|x|\geq 3$ is odd, and
$x(n)\in \lambda_n^*$ for odd $n$, $x(n)\in\lambda_n$ for even $n$, $x(|x|-1)\in\omega_1$ and whenever $n<m$ are
even such that $x(n),x(m)$ are in $\omega_1$, then there are $\beta<\alpha$ such that $x(n-1)\in S^\alpha_{n-1}$ and
$x(m-1)\in S^\beta_{m-1}$. We refer to the elements of $L_{\hbox{rel}}$ as {\emph{relevant}}.
\end{definition}

For $x\in L_{\hbox{rel}}$ let $J_x=\{z\in L: x\rest (|x|-1)\leq z< x\}$. If $x<y$ are relevant, then either $J_x\cap J_y=\emptyset$
or $J_x\subseteq J_y$. In the latter case also $|y|\leq |x|$, $x\rest (|y|-1)=y\rest(|y|-1)$ and $x(|y|-1)\leq y(|y|-1)$.

\begin{definition} Let $\mathcal{I}=\mathcal{I}(\lambda)$ be the collection of all sets of the form:
$$(\bigcup_{\alpha\in I_0} L_\alpha)\cup (\bigcup_{x\in I_1} \hbox{cl}(J_x))\cup (\bigcup_{x\in I_2}\hbox{cl}(\{x\}))
\cup (\bigcup_{x\in I_3} L_x\cap L_0),$$
where $I_0\in [\lambda_0\cup\{\lambda_0\}]^{<\omega}$, $I_1\in [L_{\hbox{rel}}]^{<}$ and $I_2,I_3$ are in $[L_1]^{<\omega}$.
\end{definition}

\begin{lemma}[Lemma 2.1,~\cite{brendle03}] $\mathcal{T}=((L,\leq), \mathcal{I}, L_0,L_1)$ is a two-sided template.
\end{lemma}

Until the end of the section, let $\PP=\PP(\mathcal{T},\QQ_{L_0},\mathbb{L})$ where $\QQ_{L_0}$ is the poset for adding a
cofinitary group with $L_0$ generators (see Definition 2.4) and $\mathbb{L}$ is the localization-forcing.

\begin{lemma}\label{ag_lowerbound} In $V^{\PP}$ there is a maximal cofinitary group of size $\lambda$ and $\lambda_0\leq\mathfrak{a}_g$.
\end{lemma}
\begin{proof} Since $L_0$ is cofinal in $L$ and $L$ is of uncountable cofinality, by Lemma 4.5 $\PP$ adds a maximal
cofinitary group of size $|L_0|=\lambda$. Since $\lambda_0\subseteq L_1$ is cofinal in $L$ and $L_\alpha\in\mathcal{I}$
for all $\alpha<\lambda_0$, by Lemma 4.1 we have $\lambda_0\leq\mathfrak{a}_g$.
\end{proof}

We say that $\dot{g}$ is a {\emph{good name}} for a real, if there are predense sets $\{p_{n,i}\}_{i\in\omega}$, where $n\in\omega$  and sets of
integers $\{k_{n,i}\}_{i\in\omega}$, $n\in\omega$ such that $p_{n,i}\Vdash\dot{g}(n)=k_{n,i}$ for all $n,i$. That is $\{p_{n,i}\}_{i\in\omega}$ is a predense set of conditions deciding the value of $\dot{g}(n)$.  Whenever $\dot{g}$ is a good name for a real, we will refer to $\bigcup_{n,i\in\omega}\hbox{dom}(p_{n,i})$ as the $L$-domain of $\dot{g}$ and denote it $\hbox{dom}_L(\dot{g})$.  We can assume that all $\PP$-names for reals are good.

The following lemma is the essence of the isomorphism of names argument, due to Brendle. Its proof follows almost identically~\cite[pages 2646-2648]{brendle03}.

\begin{lemma}[Brendle,~\cite{brendle03}]\label{isometry} Let $\lambda_0\leq\kappa<\lambda$ and for every $\beta\in\kappa$ let $B^\beta=\hbox{dom}_L(\dot{g}^\beta)$ be  countable subtree of $L$, where $\dot{g}^\beta$ is a good name for a cofinitary permutation. Then there are a countable subset $B^\kappa$ of $L$ and a good name for a cofinitary permutation $\dot{g}^\kappa$ such that
\begin{enumerate}
{\item $\Vdash_{\PP}\dot{g}^\kappa\neq\dot{g}^\beta$ for all $\beta<\kappa$,}
{\item $\hbox{dom}_L(\dot{g}^\kappa)=B^\kappa$,}
{\item for every $F\in [\kappa]^{<\omega}$ there is $\alpha<\kappa$ and a partial order isomorphism
$$\chi_{F,\alpha}:\PP_{\hbox{cl}(\bigcup_{\beta\in F}B^\beta\cup B^{\alpha})}\to \PP_{\hbox{cl}(\bigcup_{\beta\in F}B^\beta\cup B^\kappa)},$$
which maps $\dot{g}^\alpha$ to $\dot{g}^\kappa$ and fixes $\dot{g}^\beta$ for $\beta\in F$.}
\end{enumerate}
\end{lemma}

\begin{proof}[Proof of Theorem~\ref{mainthm1}] Let $\mathcal{G}$ be a $\PP$-name for a cofinitary group of size $\kappa$, where $\lambda_0\leq\kappa<\lambda$
and let $\{\dot{g}^\beta\}_{\beta\in\kappa}$ be an enumeration of $\mathcal{G}$. For $\beta<\kappa$, let $B^\beta=\hbox{dom}_L(\dot{g}^\beta)$.
Then $B^\beta$ is at most a countable subset of $L$ and without loss of generality it is a tree. Let $B^\kappa$ and $\dot{g}^\kappa$
be as in the conclusion of Lemma~\ref{isometry}, applied to the families $\{B^\beta\}_{\beta\in\kappa}$ and $\{\dot{g}^\beta\}_{\beta\in\kappa}$. We will show
that $\mathcal{H}=\la \mathcal{G}\cup\{\dot{g}^\kappa\}\ra$ is a cofinitary group.

Let $h\in\mathcal{H}\backslash\mathcal{G}$ and let $F_0\cup\{\kappa\}$  be the indexes of the permutations
involved in $h$, where $F_0\in[\kappa]^{<\omega}$. Then by Lemma~\ref{isometry}, there are $\alpha<\kappa$
and a partial order isomorphism $$\chi=\chi_{F_0,\alpha}:\PP_{\hbox{cl}(\bigcup_{\beta\in F_0}B^\beta\cup B^{\alpha})}\to \PP_{\hbox{cl}(\bigcup_{\beta\in F_0}B^\beta\cup B^\kappa)},$$
which maps $\dot{g}^\alpha$ to $\dot{g}^\kappa$ and fixes $\dot{g}^\beta$ for $\beta\in F_0$. But then $\chi^{-1}(\dot{h})$ is a name for an element of  $\mathcal{G}$ and so $|\hbox{fix}(\chi^{-1}(h))|<\aleph_0$. Since both $\PP_{\hbox{cl}(\bigcup_{\beta\in F_0}B^\beta\cup B^\alpha)}$ and $\PP_{\hbox{cl}(\bigcup_{\beta\in F_0} B^\beta\cup B^\kappa)}$ are completely embedded in $\PP$, we obtain that $V^\PP\vDash|\hbox{fix}(h)|<\aleph_0$.
\end{proof}
\end{section}

\section{Concluding Remarks}

Let $\mathcal{T}_0$ be the template used in the proof of the consistency of $\mathfrak{a}$ being of countable cofinality (see~\cite{brendle03}), the definition
of which is also stated in the previous section.

The given construction gives also a proof of the fact that the minimal size of a family of almost disjoint permutations, denoted $\mathfrak{a}_p$ can be of countable cofinality. Let $A$ be a generating set and let $\Q_A$ be the poset for adding a maximal cofinitary group defined in section 2. Let $\bar{\Q}_A$
be the suborder consisting of all pairs $(s,F)$, where every word in $F$  is of the form $ab^{-1}$ for some $a,b\in A$.
Then $\bar{\Q}_A$ is a finite function poset with the strong embedding property which adds a set of almost disjoint permutations of cardinality $|A|$, which is maximal whenever  $|A|$ is uncountable. Then  $\PP(\mathcal{T}_0,\bar{\Q}_{L_0},\mathbb{L})$ provides the consistency of $\hbox{cof}(\mathfrak{a}_p)=\omega$. The proof of maximality follows very closely the maximal cofinitary group case and the same isomorphism of names argument shows that there are no maximal families of almost disjoint permutations of intermediate cardinalities, i.e. cardinalities between $\lambda_0$ and $\lambda$. Note also that $\non(\mathcal{M})\leq \mathfrak{a}_p$.

Another relative of the almost disjointness number, which can be approached in the same way is the minimal size of a maximal almost disjoint family of functions in $^\omega\omega$.  Let $A$ be a generating set and let $\tilde{\Q}_A$ be the poset of all pairs $(s,F)$, where $s\subseteq A\times\omega\times \omega$ is finite, $s_a$ defined as above is a finite function, and $F$ is a finite set of words in the form $ab^{-1}$ for $a\neq b$ in the index set $A$. The extension relation states that $(s,F)$ extends $(t,E)$ if $s\supseteq t$, $F\supseteq E$ and for all $w\in E$ if $e_w[s](n)$ is defined and $e_w[s](n)=n$, then $e_w[t](n)=n$. Then $\P(\mathcal{T}_0,\tilde{\Q}_{L_0},\mathbb{L})$ provides the consistency of $\mathfrak{a}_e$ being of countable cofinality. Note also that to obtain a lower bound for $\mathfrak{a}_e$ in the final generic extension, we use the fact that $\non(\mathcal{M})\leq\mathfrak{a}_e$.

The consistency of  $\hbox{cof}(\mathfrak{a})=\omega$  is due to Brendle (see~\cite{brendle03}). We want to mention that his proof also fits into our general framework. More precisely, as described in Section 3, given an uncountable  generating set $A$, there is a finite function poset with the strong embedding property $\D_A$, which adds a maximal almost disjoint family of cardinality $|A|$. Then if $\D$ denote the usual Hechler forcing for adding a dominating function, the iteration $\P(\mathcal{T}_0,\D_{L_0},\mathbb{D})$ provides the consistency of $\hbox{cof}(\mathfrak{a})=\omega$.

Thus we have obtained the following statement:

\begin{theorem} Assume CH. Let $\lambda$ be a singular cardinal of countable cofinality and let  $\bar{\mathfrak{a}}\in\{\mathfrak{a},\mathfrak{a}_p,\mathfrak{a}_g,\mathfrak{a}_e\}$. Then there are a good $\sigma$-Suslin poset $\BbS_{\bar{\mathfrak{a}}}$
and a finite function poset with the strong embedding property $\Q_{\bar{\mathfrak{a}}}$, which is Knaster (and so by Lemma~\ref{l.template_knaster}  $\P(\mathcal{T}_0,\Q_{\bar{\mathfrak{a}}},\BbS_{\bar{\mathfrak{a}}})$ is Knaster)
such that $V^{\P(\mathcal{T}_0,\Q_{\bar{\mathfrak{a}}},\BbS_{\bar{\mathfrak{a}}})}\vDash {\bar{\mathfrak{a}}}=\lambda$. Then in particular $V^{\P(\mathcal{T}_0,\Q_{\bar{\mathfrak{a}}},\BbS_{\bar{\mathfrak{a}}})}\vDash\hbox{cof}(\bar{\mathfrak{a}})=\omega$.
\end{theorem}

\end{document}